\documentclass[a4paper,10pt]{article}  
\usepackage{a4wide}
\usepackage{multicol, float}
\usepackage{lineno,hyperref}
\usepackage{verbatim, euscript}
\usepackage{subfigure}
\usepackage{bbm,amssymb}
\usepackage{amsfonts}
\usepackage{amsmath, mathtools}
\usepackage{pictexwd}
\usepackage{amsthm}
\usepackage{graphicx}
\usepackage{breakcites}
\usepackage{rotating}
\usepackage{color}
\usepackage{mathrsfs}
\usepackage{mathtools}
\usepackage{titlesec}

\mathtoolsset{showonlyrefs}
\DeclarePairedDelimiter\ceil{\lceil}{\rceil}

\DeclarePairedDelimiter\floor{\lfloor}{\rfloor}

\newcommand{\eps}{\varepsilon}
\newcommand{\E}{\mathbb{E}}
\newcommand{\PP}{\mathbb{P}}
\newcommand{\1}{\mathbbm{1}} 
\newcommand{\N}{\mathbb{N}}

\newcommand{\V}{\text{Var}}
\newcommand{\EE}[1]{\mathbb{E} \left[ #1 \right]}
\newcommand{\VV}[1]{\V \left( #1 \right)}

\newcommand{\pp}[1]{\mathbb{P} \left( #1 \right)}

\newcommand{\blue}[1]{\textcolor{black}{#1}}

\theoremstyle{plain}
\newtheorem{thm}{Theorem}
\newtheorem{lem}{Lemma}[section]
  
\newtheorem{corollary}[lem]{Corollary}

\newtheorem{remark}[lem]{Remark}

\newtheorem{proposition}[lem]{Proposition}

\newtheorem{definition}[lem]{Definition}
\newtheorem{lemma}[lem]{Lemma}

\begin{document}


\title{Haldane's formula in Cannings models: \\
		The case of moderately strong selection
}


\author{
Florin Boenkost\thanks{Goethe Universit\"at, Frankfurt am Main, Germany, boenkost@math.uni-frankfurt.de} \and Adri\'an Gonz\'{a}lez Casanova\thanks{Universidad Nacional Aut\'{o}noma de M\'{e}xico, adriangcs@matem.unam.mx} \and Cornelia Pokalyuk\thanks{Goethe Universit\"at, Frankfurt am Main, Germany, pokalyuk@math.uni-frankfurt.de} \and  Anton Wakolbinger\thanks{Goethe Universit\"at, Frankfurt am Main, Germany, wakolbinger@math.uni-frankfurt.de}
}



\date{\today}

\maketitle

\begin{abstract}
For a class of Cannings models we prove Haldane's formula,  $\pi(s_N) \sim \frac{2s_N}{\rho^2}$,  for the fixation probability of a single beneficial mutant in the limit of large population size $N$ and in the regime of moderately strong selection, i.e. for $s_N \sim  N^{-b}$ and $0< b<1/2$. Here, $s_N$ is the selective advantage of an individual carrying the beneficial type, and $\rho^2$ is the (asymptotic) offspring variance. Our assumptions on the reproduction mechanism allow for a coupling of the beneficial allele's frequency process with slightly supercritical Galton-Watson processes in the early phase of fixation. 
\end{abstract}

\section{Introduction}
\label{SecIntro}

Analysing the probability of fixation of a  beneficial allele that arises from a single mutant  is one of the classical problems in population genetics, see \cite{PW} for a  historical overview.
A rule of thumb known as Haldane's formula states that the probability of fixation of a single mutant of beneficial type with small selective advantage $s > 0$ and offspring variance~$\rho^2$ in a large population of individuals, whose  total number $N$ is constant over the generations, is approximately equal to $2s/\rho^2$.  Originally, this was formulated for the (prototypical) model of Wright and Fisher, in which the next generation arises by a multinomial sampling from the previous one (which leads to $\rho^2 = 1-\frac 1N$ in the neutral case), with the ``reproductive weight'' of an individual of beneficial type being increased  by the (small) factor $1+s$. 
 A natural generalization of the Wright-Fisher model are the Cannings models; here one assumes {\em exchangeable}  offspring numbers in the neutral case (\cite{Cannings1974}, \cite{Ewens2004}), and separately within the sets of all individuals of the beneficial and the non-beneficial type in the selective case (\cite{LL}).

The reasoning in the pioneering papers by \cite{Fisher1922}, \cite{Haldane1927}  and \cite{Wright1931} was based on the insight that, as long as the beneficial type is rare, the number of individuals carrying the beneficial type is a  slightly supercritical branching process for which the survival probability is 
 \begin{align}\label{Haldanefinite}
\pi(s) \blue{\sim} \frac{2s}{\rho^2} \qquad \text{ as } s \to 0, 
\end{align}
where $1+s$ is the offspring expectation and  $\rho^2$ is the offspring variance (see \cite{Athreya}, Theorem 3). The  heuristics then is that the branching process approximation should be valid until the beneficial allele has either died out or has reached a fraction of the population that is substantial enough so that the law of large numbers dictates that this fraction should rise to 1.

Notably, \cite{LL} obtained (for fixed population size $N$) the result
 \begin{align}\label{LessLad}
\pi(s) = \frac 1N + \frac{2s}{\rho^2} + o(s) \qquad \text{ as } s \to 0, 
\end{align}
 as a special case of their explicit analytic representation of $\pi(s)$ within a quite general class of Cannings models and selection mechanisms. 
 
 An interesting parameter regime as $N\to \infty$ is that of {\em moderate selection},
\begin{align}\label{moderate}
s_N \sim c N^{-b} \quad \mbox{with } \quad 0 <b<1, \blue{ \quad c>0},
\end{align}which is between the classical regimes of {\em weak} and {\em strong} selection. Is the {\em Haldane asymptotics}
\begin{align}\label{Haldaneasymptotic}
\pi(s_N) \sim \frac{2s_N}{\rho^2}\qquad \text{ as } N \to \blue{ \infty}, 
\end{align}
valid in the regime \eqref{moderate}?

 If one could bound in this regime the $o(s)$-term in \eqref{LessLad} by $o( N^{-b})$, then \eqref{LessLad} would turn into \eqref{Haldaneasymptotic}. 
Such an estimate seems, however, hard to achieve in the analytic framework of \cite{LL}.   

The main result of the present paper is a proof of the Haldane asymptotics \label{Haldane asymptotic} using an approximation by  Galton-Watson processes in the  regime of  {\em moderately strong selection}, which corresponds to  \eqref{moderate} for $0 < b < \frac 12$. Hereby, we assume that the Cannings dynamics admits a {\em paintbox representation}, whose random weights are exchangeable and {\em of Dirichlet-type}, and fulfil a certain moment condition, see  Section \ref{Sec Main results}.  
Here, the effect of selection is achieved by a decrease of the reproductive weights of the non-beneficial individuals by the factor $1-s_N$.
  
  An approximation by Galton-Watson processes was used in \cite{GKWY} to prove the asymptotics \eqref{Haldaneasymptotic} in the regime of moderately strong selection  for a specific Cannings model that arises in the context of experimental evolution, with the next generation being formed by sampling without replacement from a pool of offspring generated by the parents. 

In the case $b \ge \frac 12$ the method developed in the present paper would fail, because then the Galton-Watson approximation would be controllable only up to a time at which the fluctuations of the beneficial allele (that are caused by the resampling) still dominate the trend that is induced by the selective advantage. However, in \cite{BoeGoPoWa1} we proved the Haldane asymptotics~\eqref{Haldaneasymptotic} for the case of {\em moderately weak selection}, i.e. under Assumption \eqref{moderate} with  $\frac 12 < b <1$.
There a backward point of view turned out to be helpful, which uses a representation of the fixation probability in terms of sampling duality via the  {\em Cannings ancestral selection graph} developed in   \cite{BoeGoPoWa1} (see also  \cite{GS}).

The results of the present paper together with those of \cite{BoeGoPoWa1} do not cover the boundary case \mbox{$b=\frac 12$} between moderately strong and moderately weak selection. We conjecture that the Haldane asymptotics \eqref{Haldaneasymptotic} is valid also in this case.

\section{A class of Cannings models with selection} \label{CanningsWithSelection}
This section is a short recap of \cite{BoeGoPoWa1} Section~2; we include it here for self-containedness.
\subsection{Paintbox representation in the neutral case}\label{Sec Paintbox}
Neutral Cannings models are characterized by the exchangeable distribution of the vector $\nu =(\nu_1,\dots, \nu_N)$ of offspring sizes; here the $\nu_i$ are non-negative integer-valued random variables which sum to $N$.
An important subclass are the {\em mixed multinomial} Cannings models. Their offspring size vector $\nu$ arises in a two-step manner: first, a vector of random weights $\mathscr{W}=(W_1,\dots,W_N)$ is sampled, which is exchangeable and \blue{ satisfies $W_1+...+W_N=1$ and $W_i\geq 0$, $1 \leq i \leq N$.}\\
In the second step, a die with $N$ possible outcomes $1,\ldots, N$ and outcome probabilities $\mathscr{W}= (W_1,\dots,W_N)$ is thrown $N$ times, and $\nu_i$ counts how often the outcome $i$ occurs. Hence, given the random weights $\mathscr{W}$ the offspring numbers $\nu = (\nu_1, ..., \nu_N)$ are Multinomial$(N, \mathscr{W})$-distributed. Following Kingman's terminology, we speak of a {\em paintbox representation} for $\nu$, and call  $\mathscr{W}$ the underlying (random) {\em paintbox}.

This construction is iterated over the generations $g\in \mathbb Z$: Let  $\mathscr{W}^{(g)}=(W_1^{(g)},\dots,W_N^{(g)})$  be independent copies of $\mathscr{W}$, and denote the individuals in generation $g$ by $(i,g)$, $i \in [N]$.   Assume that each individual $(j,g+1)$, $j \in [N] := \{1,\dots,N\}$ in generation $g+1$, \blue{chooses its parent in generation $g$, with conditional distribution}
\begin{align}
\blue{ 
\PP( (i,g) \text{ is the parent of } (j,g+1) |\mathscr{W}^{(g)})= W_i^{(g)}, \qquad \forall \, i \in [N]. }
\end{align}
where given $\mathscr{W}^{(g)}$ \blue{the choices of the parents for individuals $\{(j,g+1), j \in [N]\}$ are} independent and identically distributed. This results in  exchangeable offspring vectors $\nu^{(g)}$ which are independent and identically distributed over the generations $g$.

For notational simplicity we do not always display dependence of $\mathscr W^{(g)}$ on the generation $g$, and write $\mathscr W$ instead. From time to time however we want to emphasise the dependence of $\mathscr{W}$ on $N$ and therefore write $\mathscr{W}^{(N)}$ instead of $\mathscr{W}$.

Some exchangeable offspring  vectors do not have a paintbox representation, for example a random permutation of the vector $(2,...,2,0,...,0)$. Prototypical paintboxes are $\mathscr W = (\frac 1N, \ldots, \frac 1N)$, which leads to the Wright-Fisher model, and 
the class of Dirichlet$(\alpha,\ldots,\alpha)$-distributed random weights. In particular, the offspring distribution with Dirichlet$(1,\ldots,1)$-distributed paintbox can be seen as a limiting case of the offspring distribution for the model of experimental evolution considered in  (\cite{BGPW}, \cite{GKWY}).
\subsection{A paintbox representation with selection}
Let $\mathscr W^{(g)}$, $g\in \mathbb Z$, be as in the previous section, and let $s_N \in [0,1)$. Assume each individual carries one of two types, either the {\em beneficial type} or the {\em wildtype}. Depending on the type of individual $(i,g)$ we set
\begin{align}
\widetilde{W}_i^{(g)}= (1-s_N)W_i^{(g)}
\end{align}
if $(i,g)$ is of wildtype and $\widetilde{W}_i^{(g)}=W_i^{(g)} $ if $(i,g)$ is of beneficial type. The probability that an individual is chosen as parent is now given by
\begin{align}
\PP( (i,g) \text{ is parent of } (j,g+1) )= \frac{\widetilde{W}_i^{(g)}}{\sum_{\ell=1}^{N} \widetilde{W}_\ell^{(g)}} \label{Prob selective parent}
\end{align}
for all $i,j \in [N]$. Parents are chosen independently for all $i \in [N]$ and the distribution does not change over the generations. If $(i,g)$ is the parent of $(j,g+1)$ the child $(j,g+1)$ inherits the type of its parent.
In particular, this reproduction mechanism leads to offspring numbers that are exchangeable among the beneficial as well as among wildtype individuals.

\subsection{The Cannings frequency process}\label{CFP}
In the previous section we gave a definition for a Cannings model which incorporates selection, by decreasing the random weight of each wildtype individual by the factor $1-s_N$. This allows to define the Cannings frequency process $\mathcal{X}=(X_g)_{g\geq 0}$ with state space $[N]$ which counts the number of beneficial individuals in each generation $g$.\\
Assume there are $1 \leq k \leq N$ beneficial individuals at time $g$; due to the exchangeability of $\mathscr W^{(g)}$ we may assume that the individuals $(1,g),\ldots, (k,g)$ are the beneficial ones. Given $\mathscr W^{(g)}=\mathscr W$, the probability that individual $(j,g+1)$  is of beneficial type is then due to \eqref{Prob selective parent} equal to
\begin{align}\label{benW}
\frac{\sum_{i=1}^{k} W_i}{ \sum_{i=1}^{k} W_i+ \blue{(1-s_N )}\sum_{i=k+1}^{N}  W_i},
\end{align}
and is the same for all $j \in [N]$. Hence, given $\mathscr W^{(g)}= \mathscr W$ and given there are $k$ beneficial individuals in generation $g$, the number of beneficial individuals in generation $g+1$ has distribution
\begin{align}
\text{ Bin} \left( N, \frac{\sum_{i=1}^{k} W_i}{ \sum_{i=1}^{k} W_i + \blue{(1-s_N )} \sum_{i=k+1}^{N} W_j} \right);\label{Transition probabilities beneficial}
\end{align}
this defines the transition probabilities of the Markov chain $\mathcal{X}$.

\section{Main Result}\label{Sec Main results}
Before we state our main result we specify the assumptions on the paintbox and the strength of selection.
\begin{definition}(Dirichlet-type weights) \label{Def Dirichlet weights}
	We say that a random vector $\mathscr W^{(N)}$ with exchangeable components $W_1^{(N)},\ldots, W_N^{(N)}$  is {\em of Dirichlet-type} if 
	\begin{align}
	W_i^{(N)}=\frac{Y_i}{\sum_{\ell=1}^{N} Y_\ell}, \quad i=1,\ldots, N, \label{Dirichlettypeweights}
	\end{align}
	where $Y_1, \ldots, Y_N$ are independent copies of a random variable $Y$ with $\mathbb P(Y > 0) =1$.
\end{definition}
\noindent We assume that
\begin{align} 
\EE{\exp(hY)} < \infty, \label{Generating function Y}
\end{align}
for \blue{some $h>0$}, which implies the finiteness of all moments of $Y$. \blue{The relevance (and possible relaxations) of Condition \eqref{Generating function Y} are discussed further in Remark \ref{ConditionExpMoments} a), see also the comment in \mbox{Remark \ref{Remark1} a).}}\\ 

\begin{remark}\label{Remark1}
	\begin{itemize}
	 \item[a)] The biological motivation for considering Dirichlet-type weights comes from seasonal reproductive schemes. At the beginning of a season a set (of size $N$) of individuals is alive. These individuals  and their offspring reproduce and generate a pool of descendants within that season. Only a few individuals from this pool survive till the next season. The number $N$ in the model is assumed to be the total number of individuals that make it to the next season. Dirichlet-type weights arise in the asymptotics of an infinitely large pool of offspring; then sampling with and without replacement coincide. 
	    Condition \eqref{Generating function Y}, which we will require for the proof of Theorem \ref{HaldaneThm} (see also Remark \ref{ConditionExpMoments}), guarantees that the pool of descendants of a single individual is not too large in comparison to the pool of descendants generated by the other individuals. The simplifying assumption  $\mathbbm{P}(Y>0)=1$ implies that the weight $W_i^{(N)}$ of a parent cannot be equal to zero. Observe, however, that weights of single parents can be arbitrarily small if (e.g.) $Y$ has a density which is continuous and strictly positive in zero.
		\item[b)] \blue{ 
		The case of a deterministic $Y$ corresponds to $W_i^{(N)} \equiv 1/N$, i.e. the classical Wright-Fisher model. If $Y$ has a Gamma($\kappa$)-distribution, then $\mathscr W^{(N)}$ is \mbox{Dirichlet $(\kappa,\ldots, \kappa)$}-distributed.\\Theorem 1 in \cite{Huillet2021}, gives a classification of a large class of Cannings models with a paintbox of the form \eqref{Dirichlettypeweights} with regard to the convergence of their rescaled genealogies.
		\item[c)]  Let $\nu^{(N)}$ be a sequence of Cannings offspring numbers that are represented by the paintboxes $\mathscr W^{(N)}$. It is well known (and easily checked) that
		\begin{align}\label{Varnu}
		 \VV{\nu_1^{(N)}} = N(N-1)\mathbb E[(W_1^{(N)})^2].
		 \end{align}
If $\mathscr W^{(N)}$ is of the form  \eqref{Def Dirichlet weights} with $\mathbb E[Y^2] < \infty$, which is clearly implied by~\eqref{Generating function Y},  then (see  \cite{Huillet2021} Theorem 1 (i))
	the right hand side of \eqref{Varnu} converges to $\frac{\mathbb E[Y^2]}{\mathbb E[Y]^2}$  as $N\to \infty$. }
 \end{itemize}
\end{remark}
\blue{In view of  Remark \ref{Remark1} c) we have for the asymptotic neutral offspring variance
\begin{align}\label{rho2}
 \lim\limits_{N\to \infty} \VV{\nu_1^{(N)}} =: \rho^2= \frac{\mathbb E[Y^2]}{\mathbb E[Y]^2}.
\end{align}
Replacing  $Y$ by  $Y' = \frac{Y}{\mathbbm{E}[Y]}$ does not affect   \eqref{Dirichlettypeweights}, hence we can and will assume $\E[Y] = 1$ in the proofs, which simplifies \eqref{rho2} to $\rho^2 = \E[Y^2]$ (and makes \eqref{rho2} consistent with the notation of \cite{Huillet2021}).
Under the assumption $\mathbb E[Y^4] < \infty$, which is implied by~\eqref{Generating function Y} as well, the following asymptotics is valid
\begin{align}
\VV{\nu_1^{(N)}}=\rho^2 +O(N^{-1}), \mbox{ as } N\to \infty\label{second moment W}.
\end{align}
}
We will discuss the relevance of this asymptotics in Remark \ref{ConditionExpMoments} b), and prove it in Lemma \ref{LemmaSecondMoment}~b).
\\\\
 \noindent Turning  to the selective advantage, we assume that for a fixed $\eta \in (0,\frac 14)$ the sequence $(s_N)$ obeys
\begin{align}
N^{-\frac{1}{2} +\eta} \leq s_N \leq N^{-\eta} \label{Condition s_N},
\end{align}
which we call the regime of {\em moderately strong selection}, thus generalizing the corresponding notion introduced in Section \ref{SecIntro}. \blue{  (Note that \eqref{Condition s_N} has an analogue in the regime of \emph{moderately weak selection} as discussed in \cite{BoeGoPoWa1}).}
In order to connect to \eqref{moderate} we define 
\begin{align}\label{DefbN}
 b_N := -\frac{ \ln s_N}{\ln N}
\end{align}
which is equivalent to  $s_N = N^{-b_N}$, with \eqref{Condition s_N} translating to 
\begin{align*}
  \eta \le b_N \le \frac{1}{2} - \eta.
\end{align*}

We now state our main result  on the asymptotics (as $N\to \infty$) of the fixation probability  of the Cannings frequency process $(X_g^{(N)}) = (X_g)$ defined in  Subsection \ref{CFP}. Note that the Markov chain $(X_g)$ has the two absorbing states $0$ and $N$, with the hitting time of $\{0,N\}$ being a.s. finite \mbox{for all $N$.}

\begin{thm}\label{HaldaneThm}(Haldane's formula)
\qquad \\
	\emph{Assume that Conditions \eqref{Dirichlettypeweights}, \eqref{Generating function Y} and \eqref{Condition s_N} are fulfilled. Let $(X_g)_{g \geq 0}$ be the number of beneficial individuals in generation $g$, with $X_0= 1$. Let $\tau=\inf \left\lbrace g \geq 0 : X_g \in \left\lbrace 0,N\right\rbrace \right\rbrace $, then}
	\begin{align}
	\PP(X_{\tau}=N) \blue{ \sim \frac{2 s_N}{\rho^2},\qquad  \text{ as } N \to \infty.} \label{Haldane}
	\end{align}
\end{thm}
We give the proof of Theorem \ref{HaldaneThm} in Section \ref{SecProofofTheorem}, after preparing some auxiliary results in Section \ref{Sec Aux}. Next we give a strategy of the proof and its main ideas, with an emphasis on the role of Condition~\eqref{Condition s_N}. In Remark \ref{ConditionExpMoments} we discuss possible relaxations of Condition \eqref{Generating function Y} and the boundary case $b=\frac{1}{2}.$\\\\
The proof of Theorem \ref{HaldaneThm} is divided into three parts, corresponding to three growth phases of $\mathcal X$. Concerning the \emph{first phase} we show that the probability to reach the level $N^{b+\delta}$ is $\frac{2 s_N}{\rho^2}(1+o(1))$, for some small $\delta >0$ \blue{and $b:=b_N$}; this is the content of Proposition \ref{Proposiotion reaching critical}. The proof is based on stochastic domination from above and below by slightly supercritical Galton-Watson processes $\overline{ \mathcal{Z} }$ and $\underline{ \mathcal{Z} }$ with respective offspring distributions \eqref{offspringupper} and \eqref{offspringlower}.
\\\\
To construct a Galton-Watson stochastic upper bound $\overline{ \mathcal{Z} }$ of $\mathcal X$ in its initial phase, we recall that the transition probabilities of $\mathcal X$ are mixed Binomial specified by \eqref{Transition probabilities beneficial}. Using \eqref{Dirichlettypeweights} we approximate~\eqref{benW} from above
by 
\begin{align}\label{approxindep} 
 \frac{ 1+s_N +o(s_N) }{N }\sum_{\ell=1}^{k} Y_\ell.
\end{align}
As we will show in Lemma \ref{Lemma Independent Weights}, this is possible with  probability $1-O(\exp(-c'N^{1-2\alpha}))$ for some  $\alpha < \frac 12$ and $c'>0$, and
for $k \leq N^{b+\delta}$ with $b+\delta< 1/2$. 
We will then be able to dominate the mixed Binomial distribution \eqref{Transition probabilities beneficial} by the mixed Poisson distribution with random parameter \eqref{approxindep}, again up to an
error term of order $o(s_N)$. Noting that \eqref{approxindep} is a  sum of independent random variables,
we arrive at the upper Galton-Watson approximation for a single generation. For any small $\eps>0$ this can be repeated for $ N^{b+\eps} $ generations, which (as an application of Lemma \ref{Lemma ext time GWP}	 will show) is enough to reach either the level $0$ or the level $N^{b+\delta}$  with probability $1-o(s_N)$.
\\\\
To obtain a  Galton-Watson stochastic lower bound $\underline{ \mathcal{Z} }$ of $\mathcal X$ in its initial phase, we adapt an approach that was used in
\cite{GKWY} in a related  situation. As in Subsection \ref{Sec Paintbox}, number the individuals in generation $g$ by $(i,g)$, now with $(1,g), \ldots, (X_g,g)$,  being the beneficial individuals, and
denote by $\omega_i^{(g)}$ the number of children of the individual $(i,g)$, $1\le i \le X_g$.
As will be explained in the proof of Lemma \ref{Lemma Coupling with GWP}, as long as $X_g$ has not reached the level $N^{b+\delta}$, the distribution of  $\omega_i^{(g)}$ can be bounded from below by a mixed binomial distribution
\begin{align}
 \text{ Bin} \left( N - \lceil N^{b+\delta} \rceil , Y_1 \frac{1+s_N +o(s_N)}{N}\right)
\end{align}
with probability $1-O(\exp(-N^{\eps}))$ for some sufficiently small $\eps>0$, again for $b + \delta <1/2$.  A suitable stopping and truncation at the level $N^{b+\delta}$ will give the Galton-Watson process approximation from below for the first phase.\\\\
We will verify in Subsection \ref{SsGW} that both slightly supercritical branching processes $\underline{\mathcal Z}$ and $\overline{ \mathcal{Z} }$ reach the level $N^{b+\delta}$ with probability $\frac{2s_N}{\rho^2} (1+o(1))$.
\\\\
As to the \emph{second phase}, we will argue in Section \ref{Sec Second Phase} that, after reaching the level $N^{b+\delta}$ the Cannings frequency process $\mathcal X$ will grow to a macroscopic fraction $\eps N$ with high probability. If the frequency of beneficial individuals is at least $N^{b+\delta}$ (but still below $\eps N$), then in a single generation the frequency of beneficial individuals grows in expectation at least by $1 + (1- \eps) s_N  + o(s_N)$. Hence, $c s_N ^{-1} \ln N$ generations  after $\mathcal X$ has reached the level $N^{b+\delta}$, the expected value of the process $\mathcal{X}$ reaches the level $2 \eps N$. 
Similarly one bounds the variance produced in a single generation and derives from this an estimate for the variance accumulated over $c s_N ^{-1} \ln N$ generations. This bound being sufficiently small, an application of Chebyshev's inequality yields that (after $c s_N ^{-1} \ln N$ generations) $\mathcal{X}$ crosses the level $\eps N$ with probability tending to $1$ after reaching the level $N^{b+\delta}$.
\\\\
In Section \ref{Sec Third Phase}) we deal with the \emph{last phase}, and will  show that the fixation probability tends to $1$ as $N \to \infty$ if we start with at least $\eps N$ individuals of beneficial type. Here we use the representation for the fixation probability that is based on a sampling duality between the Cannings frequency process and the Cannings ancestral selection process (CASP) which was provided in \cite{BoeGoPoWa1}. 
For a subregime of moderately weak selection the claim will follow quickly  from the representation formula  combined with a concentration result  for the equilibrium distribution of the CASP that was proved in \cite{BoeGoPoWa1}. To complete the proof we will then argue that both the CASP and the representation of the fixation probability depend on the selection parameter in a monotone way.
 
\begin{remark}\label{ConditionExpMoments}
\begin{itemize}
\item[a)] \blue{ With some additional work the assumption \eqref{Generating function Y} of the existence of some exponential moment of $Y$ can be relaxed to some weaker moment condition. 
In order not to overload the present paper, we restrict here to a sketch.}

\blue{ In Lemma \ref{Lemma Coupling with GWP} we couple the frequency process of the beneficial individuals with Galton-Watson processes for $N^{b+\delta}$ generations. By means of the estimates in Lemma \ref{Lemma Large Deviations Weights} and Lemma \ref{Lemma Independent Weights} we show that these couplings hold for a single generation with probability $1-O(\exp(-N^{c'}))$ for some appropriate $c'>0$. Since we need the couplings to hold for $N^{b+\delta}$ generations, it suffices that the couplings hold in a single generation with probability $1 - O(N^{-2(b + \delta)})$ for some $\delta>0$ (since in this case the probability of the coupling to fail is $o(s_N)$ and therefore can be neglected with regard to \eqref{Haldane}). }
\blue{
Such probability bounds can also be obtained under weaker assumptions on the distribution of the random variable $Y$. Assume e.g. that $Y$ has a regularly varying tail, i.e. $\mathbbm{P} (Y >x)\sim x^{-\beta} L(x)$ for some $\beta >0$ and $L$ is a slowly varying function. For the proof of Lemma~\ref{Lemma Coupling with GWP} we need to estimate the probability of the event figuring in Lemma \ref{Lemma Large Deviations Weights} with $b<c \leq 1$ and the probability of the event figuring in Lemma \ref{Lemma Independent Weights} with $b<\alpha <\frac{1}{2}$. To show that these probabilities are of order $O(N^{-2(b + \delta)})$ we  only need that $\mathbbm{P}\left(\sum_{i=1}^n Y_i > x \right) = O(n^{-2(b + \delta)})$ (since the remaining probability in Lemma \ref{Lemma Large Deviations Weights} can be estimated with Hoeffding's inequality, see \cite{Hoeffding1994}) with $n = N, x = N^{1-\alpha}$ in Lemma \ref{Lemma Independent Weights} and $n= N^c, x= N^{c}$ in Lemma \ref{Lemma Large Deviations Weights}. The asymptotics~(3.2) in \cite{MikoschNagaev1998} states that
$\mathbbm{P}\left(\sum_{i=1}^n Y_i > x \right) \sim n x^{-\beta} L(x)$. Consequently, we need to choose $\beta>0$ such that $N^{1 - \beta(1-\alpha)} L(N^{1-\alpha}) = O(N^{-2(b + \delta)})$ as well as $N^{c - \beta c} L(N^c) = O(N^{-2(b + \delta)})$. This works for all choices of $0< b< \tfrac{1}{2}$, provided that $\beta \geq 4$.}

\blue{It would be nice to have a proof of the asymptotics \eqref{Haldane} under the assumption that the 4th moment of $Y$ is finite, even without the assumption of a regularly varying tail.}

\blue{ The investigation of the analogue to \eqref{Haldane} in the absence of finite second moments, i.e. for Cannings models with heavy-tailed offspring distributions, is the subject of ongoing research, and will be treated in a forthcoming paper. }
\item[b)] \blue{ Relation \eqref{second moment W} will be used in the proof of Lemma \ref{Lemma Hitting epsN}. Moreover, this relation is also instrumental  in the companion paper \cite{BoeGoPoWa1} (on the regime of moderately weak selection). The special case $n=3$ in Lemma \ref{LemmaSecondMoment} a) shows that the assumption $\mathbb E[Y^3] <\infty$ implies}
\begin{align}\label{Wthird}
\blue{ \mathbb E[(W_1^{(N)})^3] = O(N^{-3})}.
\end{align}
\blue{This gives a rate of decay $O(N^{-2})$ for the triple coalescence probability (and is the moment condition (3.6) in \cite{BoeGoPoWa1}).} 

\blue{Condition \eqref{Generating function Y}  (on the existence of an exponential moment of $Y$)  guarantees the Haldane asymptotics \eqref{Haldane} for Cannings models with weights of Dirichlet type also in the whole regime of moderately weak selection $N^{-1 +\eta} \leq s_N \leq N^{-\frac{1}{2} -\eta}$ without any further assumption. In particular the assumption  on the finiteness of a negative moment of $Y$ in \cite{BoeGoPoWa1}, Lemma 3.7 b), is unnecessary. Indeed, in the proof of Lemma 
\ref{LemmaSecondMoment} a) we show that $\mathbbm{E}[(W_1^{(N)})^n] \leq  \left(\frac{2}{N}\right)^n( \mathbbm{E}[Y^n] + o(1))$. As shown in the proof of Lemma 3.7 b) in \cite{BoeGoPoWa1} Condition \eqref{Generating function Y} guarantees that for a sequence $(h_N)$ with $h_N \rightarrow \infty$ and $h_N \in O(\log N)$ for all $n\leq  2 h_N$ we can estimate $\mathbbm{E}[Y^n]$ from above by $C \left(\frac{2h_N}{c}\right)^n$ for appropriate constants $C,c>0$. Consequently, for $N$ sufficiently large we have
$\mathbbm{E}[(W_1^{(N)})^n] \leq \left(\frac{Kh_N}{N} \right)^n$ for some appropriate constant $K>0$, that is Condition (3.8) in \cite{BoeGoPoWa1} is fulfilled.}
\item[c)]
\blue{It seems a mathematically intriguing question whether in the regime of moderate selection all Cannings models which admit a paintbox representation with Dirichlet-type weights and are in Kingman domain of attraction, also follow the Haldane asymptotics \eqref{Haldane}.}

\blue{An example of a sequence of Cannings models (with weights {\em not} of Dirichlet-type) which fulfil M\"ohle's condition but do not follow the Haldane asymptotics, is the following. In each generation a randomly chosen individual gets weight $N^{-\gamma}$, $0<\gamma <\frac{1}{2}$ and all the other individuals have a weight of $\frac{1-N^{-\gamma}}{N-1}$. Then we have $\EE{W_1^2}\sim N^{-1-2\gamma}$ and $\EE{W_1^3}=o(\EE{W_1^2})$, therefore by M\"ohle's criterion the genealogy lies in the attraction of Kingman's coalescent. However, the Haldane asymptotics would predict that the survival probability is of order $s_N/(N^2 N^{-1-2\gamma})\sim N^{-1-b+2\gamma}$, which for $\gamma < b/2$ is $\ll N^{-1}$. Since the  fixation probability of a beneficial allele cannot be smaller than the fixation probability under neutrality (which is $\frac{1}{N}$), \eqref{Haldane} must be violated in this example.}
\item[d)] \blue{ The present work together with the approach in \cite{BoeGoPoWa1} does not cover the boundary case $b=\frac{1}{2}$. A quick argument why our arguments cannot  be extended simply to the boundary case is the following. We show that once the beneficial type exceeds (in the order of magnitude) the frequency $s_N^{-1}= N^b$ it goes to fixation with high probability. In the regime $b<\frac{1}{2}$ we use couplings with Galton-Watson processes to show that this threshold is reached with probability $\frac{2s_N}{\rho^2}(1+ o(1))$. However, these couplings are not guaranteed as soon as collisions occur, i.e.  when beneficial individuals are replacing beneficial individuals. By the well known ''birthday problem`` collisions are common as soon as $N^{\frac{1}{2}}$ individuals are of the beneficial type. Therefore we require $N^{b} \ll N^{\frac{1}{2}}$, i.e. $b <\tfrac{1}{2}$.} 

\blue{
In the light of the
results of the present paper and of \cite{BoeGoPoWa1}, there is little reason to
doubt that the assertion of Theorem \ref{HaldaneThm} should fail in the boundary
case $b=1/2$. However, the question remains open (and intriguing)
whether then the backward or the forward approach (or a combination
of both) is the appropriate tool for the proof.}
\end{itemize}
\end{remark}

\section{Auxiliary results} \label{Sec Aux}
\subsection{Slightly supercritical Galton-Watson processes} \label{SsGW}
Throughout this subsection, $(s_N)_{N\in \mathbbm{N}}$ is a sequence of positive numbers converging to $0$, $\sigma^2$ is a fixed positive number, and  $Z^{(N)}=(Z_n^{(N)})_{n\geq 0}$, $N=1, 2, \ldots$ are Galton-Watson processes with offspring expectation 
\begin{align}\label{offex}
\mathbb E_1[Z_1^{(N)}] = 1+s_N+o(s_N),
\end{align}
offspring variance $\sigma^2 +o(1)$ and uniformly bounded third moments $\mathbb E_1[(Z_1^{(N)})^3]$. Unless stated otherwise we assume that  $Z_0^{(N)}=1$. We write 
\begin{align}\label{survprob}
\phi_N := \mathbb P\left(\lim_{n\to \infty} Z_n^{(N)}=\infty\right) = 1-\mathbb P\left(Z_n^{(N)} = 0 \mbox{ for some } n > 1\right)
\end{align}
 for the survival probability of $(Z^{(N)})$ and observe 
 \begin{align}\label{equaiton surv prob GW}
	\phi_N\blue{\sim} \frac{2 s_N}{\sigma^2}.
	\end{align}
	The derivation and discussion of the asymptotics \eqref{equaiton surv prob GW} has a venerable history, a few key references being \cite{Haldane1927}, \cite{Kolmogorov1938}, \cite{Eshel1981}, \cite{Hoppe1992}, \cite{Athreya} Theorem~3, \cite{Haccou2005Branching} Theorem 5.5.

	 Lemma B.3 in \cite{GKWY} gives a statement on the asymptotic probability that $Z^{(N)}$ either quickly dies out or reaches a certain (moderately) large threshold.  The following lemma improves on this in a twofold way. It dispenses with the assumption  $s_N \sim c N^{-b}$ for a fixed $b\in (0,1)$ and more substantially, it gives a {\em quantitative} estimate for the probability that, given non-extinction, the (moderately) large threshold is reached quickly.

  \begin{lemma}\label{Lemma ext time GWP}		
Fix $\delta>0$, and let $T^{(N)}:=\inf \{n\geq 0 : Z_n^{(N)} \notin \{1,2,...,\lceil (\frac{1}{s_N} )^{1+\delta} \rceil  \}$. Then, for all  $\eps>0$
	\begin{align}
	\PP_1\left( T^{(N)}>
	(1/s_N)^{(1+\eps)}  \right)= O\left(\exp \left( -cs_N^{- \eps/2}  \right) \right), \label{Time to fixation}
	\end{align}
	with $c = - \log\left(\tfrac{7}{8}\right)$.
\end{lemma}
\begin{proof}
Observe that
	\begin{align}
	\PP_1\left(  T^{(N)}> \left(1/s_N\right)^{1+\eps} \right) &\leq \PP_1\left( T^{(N)}> (1/s_N)^{1+\eps} \Big{|} Z^{N} \text{ survives} \right)\\
	& \qquad +\PP_1\left( T^{(N)}> (1/s_N)^{1+\eps}\Big{|} Z^{N} \text{ dies out}  \right). \label{split of probability}
	\end{align}
	In Part 1 of the proof we will estimate the first probability on the r.h.s. of \eqref{split of probability}; this will give the above-mentioned improvement of Lemma B.3 in \cite{GKWY}. Part 2 of the proof deals with the second probability on the r.h.s. of \eqref{split of probability}.

{\em Part 1. }	
	 Like in the proof of Lemma B.3 in \cite{GKWY} we obtain an upper bound on the time at which the process $Z^{(N)}$ reaches the level $(1/s_N)^{1+\delta}$ given survival, by considering the process $Z^{\star}=(Z_n^{\star})_{n\geq 0}$ consisting of the immortal lines of $Z^{(N)}$ conditioned to non-extinction. (For simplicity of notation we drop a superscript $N$ in $Z^{\star}$.) Let  $\phi_N$ denote the survival probability of $Z^{(N)}$ as in \eqref{survprob}. \blue{The offspring distribution of $Z^{\star}$ arises from that of $Z^{(N)}$  as  
	 	\begin{align}
	 \PP_1(Z_1^{\star} = k) = \frac 1{\phi_N^{}} \mathbb E_1\left[{{Z_1^{(N)}}\choose{k}}\phi_N^k (1-\phi_N)^{Z_1^{(N)}-k}\right]  , \quad k \ge 1. \label{Z star = k}
	\end{align}
	(see \cite{Lyons2017} Proposition 5.28). In particular one has
\begin{align*}
	\mathbb E_1[Z_1^{\star}] = \frac 1{\phi_N} \mathbb E_1[Z_1^{(N)}\phi_N ] = \mathbb E_1[ Z_1^{(N)}].
\end{align*}
	Denote, as usual, for a random variable $X$ and an event $A$ by $\EE{X;A}:=\EE{X \1_A}$. Furthermore,
\begin{align*}
	\mathbb E_1[Z_1^{\star}; Z_1^{\star}\ge 3] \le \frac 1{\phi_N^{}} \mathbb E_1\left[{{Z_1^{(N)}}\choose{3}}\phi_N^3\right] = O(\phi_N^2)
\end{align*}	
	because of the assumed uniform boundedness of the third moments of $Z_1^{(N)}$. These two relations together with \eqref{offex}, \eqref{equaiton surv prob GW} and the fact that $\mathbb E_1[Z_1^{\star}; Z_1^{\star}= 1]\le 1$ immediately give a lower bound for $\mathbb E_1[Z_1^{\star}; Z_1^{\star}= 2]\le 1$, implying that for any $\beta \in (0,1)$
	\begin{align}
	\PP_1 ( Z_1^{\star} \geq 2) \geq \beta s_N , \qquad \PP_1(Z_1^{\star} =1) \leq 1-\beta s_N. \label{Probs Z star}
	\end{align}	
	Hence the process $Z^{(N)}$, when conditioned on survival, is bounded from below by the counting process $Z^{\star}$ of immortal lines, which in turn is bounded from below  by the process $\widetilde{Z}=(\widetilde{Z}_n ) _{n \geq 0 }$ with offspring distribution 
	\begin{align}
	\nu = (1-\beta s_N) \delta_1 + \beta s_N \delta_2.
	\end{align}}
So far we closely followed the proof in \cite{GKWY}, but now we deviate from that proof to obtain the rate of convergence claimed in \eqref{Time to fixation}.	

An upper bound for the time $\widetilde{T}:= \inf \{n \geq 0 : \widetilde{Z}_n \geq (1/s_N)^{1+\delta} \}$ also gives an upper bound for the time $T^{(N)}$. The idea is now to divide an initial piece of  $k \le (1/s_N)^{(1+\eps)}$ generations into  $\floor{(1/s_N)^{\eps/2}}$  parts, each of  $n_0 \le (1/s_N)^{(1+\eps/2)}$ generations.  Because of the immortality of $\widetilde{Z}$ and the independence between these parts we obtain immediately that
\begin{align}\label{keybound} \mathbb P(\widetilde T \ge (1/s_N)^{(1+\eps)}) \le   \mathbb P_1(\widetilde{Z}_{ j } \le (1/s_N)^{1+\delta} \mbox { for } j=1,\ldots, k ) \le \left(\mathbb P_1(\widetilde{Z}_{n_0}\le (1/s_N)^{1+\delta} )\right)^{\floor{(1/s_N)^{\eps/2}} }
\end{align}
We then bound $\PP_1 ( \widetilde{Z}_{n_0}   >(1/s_N)^{1+\delta} )$ from below by an application of the Paley-Zygmund inequality in its form
\begin{align}\label{PZ}
\mathbb P\left(X\ge \frac{\mathbb E[X]}2\right) \ge \frac 14 \frac {(\mathbb E[X])^2}{\mathbb E[X^2]},
\end{align}
 where $X$ is a non-negative random variable (with finite second moment). For a supercritical Galton-Watson process with offspring expectation $m$ and offspring variance $\sigma^2$ the $n$-th generation offspring expectation and $n$-th generation offspring variance $\sigma^2_n$ are given by  $m^n$ and\\ $\sigma^2 m^n (m^n -1)/(m^2 -m)$ (see \cite{Athreya1972}, p.4). Hence, we obtain 
\begin{align}
	\mathbb E_1[\tilde Z_n] = (1+ \beta s_N)^n , \qquad {\rm Var}_1[\tilde Z_n] = \frac{\beta s_N (1-\beta s_N) (1+\beta s_N)^n ((1+  \beta s_N )^n -1) }{(1+\beta s_N)^2 -(1+\beta s_N)}. \label{Var+E of GWP}
\end{align}
\blue{	We choose the smallest $n_0$ such that }
	\begin{align}\label{choiceofn0}
	\mathbb E_1{ [\widetilde{Z}_{n_0}] } \blue{ \geq } 2 (1/s_N)^{1+\delta}.
	\end{align}
	 Observe that
	$n_0 \sim \frac{1}{\beta s_N} \log(2 (\frac{1}{s_N})^{1+\delta}  )  $ which ensures that $(1/s_N)^{\eps/2} n_0 \leq (1/s_N)^{1+\eps}$ for $N$ large enough. 
We now estimate $\mathbb E_1{ [(\widetilde{Z}_{n_0})^2] }$  using \eqref{Var+E of GWP} as follows
	\begin{align}
	\EE{\widetilde{Z}_{n_0}^2} &= \frac{\beta s_N (1-\beta s_N) (1+\beta s_N)^{n_0} ((1+  \beta s_N )^{n_0} -1) }{(1+\beta s_N)^2 -(1+\beta s_N)} + (1+\beta s_N)^{2n_0} \\
	&\leq \frac{\beta s_N  (1+\beta s_N)^{2n_0} }{ \beta s_N + (\beta s_N )^2 } + (1+\beta s_N)^{2n_0} \leq 2 (1+\beta s_N)^{2n_0}. \label{Second Moment estimate GWP}
	\end{align}
	Applying \eqref{PZ} with $X:= \widetilde Z_{n_0}$ yields
	\begin{align}
	\PP_1 \left( \widetilde{Z}_{n_0} \geq \frac{1}{2} \EE{\widetilde{Z}_{n_0}} \right)  \geq \frac{1}{4} \frac{(1+\beta s_N)^{2n_0} }{2 (1+\beta s_N)^{2n_0}}= \frac{1}{8}, \label{PZ application}
	\end{align}
	which because of \eqref{choiceofn0} implies $\PP_1( \widetilde{Z}_{n_0} \leq (1/s_N)^{1+\delta} ) \leq \frac{7}{8}$. If after time $n_0$ the process $\widetilde{Z}_{n_0}$ is still smaller than our desired bound $ (1/s_N)^{1+\delta} $, we can iterate this argument $\floor{(1/s_N)^{\eps/2}}$ times and arrive at
	\begin{align}
	(\PP_1( \widetilde{Z}_{n_0} \leq (1/s_N)^{1+\delta} )^{\floor{(1/s_N)^{\eps/2}}}
	\leq \left( \frac{7}{8}\right)^{\floor{(1/s_N)^{\eps /2}}} = \exp(-c \floor{(1/s_N)^{\eps/2} }),
	\end{align}
	with $c = -\log \frac{7}{8}$.
	This gives the desired bound for the first term in \eqref{split of probability}.
	
	{\em Part 2.} We now turn to the second term on the r.h.s. of \eqref{split of probability}.
	Define 
	\begin{align*}T_0^{(N)}:= \inf \{ n \geq 0 : Z^{(N)}_n=0\}.
	\end{align*}
	 Obviously $T^{(N)}\leq T_0^{(N)}$, and so it suffices to prove
	\begin{align}
	\PP\left(T_0^{(N)}> (1/s_N)^{1+\eps}|Z^{(N)} \text{ dies out}\right) \leq \exp(-\beta s_N^{-\eps} (1+o(1))). \label{ProvebyMarkov}
	\end{align}
	This proof follows closely that of the second part of Lemma B.3 in \cite{GKWY}; we include it here for completeness. 
	
	We observe
	\begin{align}
	\E_1\left[ Z_1^{(N)}| Z^{(N)} \text{ dies out} \right] & =  \frac 1{1-\phi_N} \E_1\left[ (1-\phi_N)^{Z_1^{(N)}} Z_1^{(N)}\right]  =  \E_1\left[ (1-\phi_N)^{Z_1^{(N)}-1} Z_1^{(N)}\right] \\ & = \PP_1(Z_1^{\star} =1),
	\end{align}
	where the first and the last equality follow from the branching property and  from \eqref{Z star = k}, respectively.
	We have shown in \eqref{Probs Z star} that $\PP_1(Z_1^{\star} =1) \leq 1- \beta s_N +o(s_N)$ and hence we can conclude
	\begin{align}
	\EE{Z_{\floor{ (1/s_N)^{1+\eps}}}|Z^{(N)} \text{ dies out } } \leq  (1-\beta s_N +o(s_N))^{\floor{(1/s_N)^{1+\eps}}} \leq \exp(-\beta s_N^{-\eps} (1+o(1))).
	\end{align}
	Finally, an application of Markov's inequality yields \eqref{ProvebyMarkov}
	\begin{align}
	\PP(T_0^{(N)}> (1/s_N)^{1+\eps}|Z^{(N)} \text{ dies out }) \leq \PP(Z_{\floor{(1/s_N)^{1+\eps}}} \geq 1|Z^{(N)} \text{ dies out }) \leq \exp(-\beta s_N^{-\eps} (1+o(1))).
	\end{align}
\end{proof}
\subsection{Estimates on the paintbox}
The following lemma provides the asymptotics \eqref{second moment W} as well as the moment bounds for the Dirichlet-type weights that were addressed in Remark \ref{ConditionExpMoments} b).
\begin{lemma}[Moments of the weights]\label{LemmaSecondMoment}
 \blue{Let $Y, Y_1, ..., Y_N$ be iid positive random variables with $\mathbbm{E}[Y]=1$ and $\rho^2 = \mathbbm{E}[Y^2].$ We abbreviate $H_N:=Y_1+\cdots+Y_N$.\\
 a) Assume $\mathbbm{E}[Y^n] < \infty$ for some $n\in \mathbbm{N}$. Then }
 \begin{align*}
  \blue{ \mathbbm{E}\left[\left(\frac{Y}{(Y +H_N)}\right)^n \right] = O(N^{-n}).}
 \end{align*}
\noindent b) \blue{  Assume $\mathbb E[Y^4] < \infty$.
 Then
 \begin{align*}
  \blue{\mathbbm{E}\left[\left(\frac{Y}{(Y +H_N)}\right)^2 \right] = \frac{\rho^2}{N^2}  + O(N^{-3}).}
 \end{align*}
 }
\end{lemma}
\begin{proof} \blue{ 
a) Consider the event $F_N:= \{\frac {H_N}{N} \le \frac 12\}$. First we note that
\begin{align}
	\EE{\left(\frac{Y}{(Y +H_N)}\right)^n } \leq \EE{\left(\frac{Y}{Y+H_N}\right)^n\1_{F_N}}+\EE{\left(\frac{Y}{H_N}\right)^n\1_{F_N^c}} \leq \pp{F_N}+ \left(\frac{2}{N}\right)^n \EE{Y^n}.
\end{align}
Let $K$ be so large that $\mathbb E[Y\wedge K] > \frac 12$. Hoeffding's inequality applied to the sample mean of i.i.d. copies of the bounded random variable $Y \wedge K$ implies that $\mathbb P(F_N)$ decays exponentially fast. Since $\EE{Y^n}$ is bounded this yields the claim.}\\
b)
\blue{
We observe that
 \begin{align*}
 \mathbbm{E}\left[\frac{Y^2}{(Y + H_N)^2} \right] &= \frac{1}{N^2} \mathbbm{E}\left[\frac{Y^2}{\left( \frac{Y}{N} +\frac{H_N}{N} \right)^2 } \right].
 \end{align*}
Let $E^{(1)}_N: = \{\frac{H_N }{N} > \frac{5}{4}\}, \, E^{(2)}_N := \{ \frac{H_N}{N} < \frac{3}{4} \}$ and $E^{(3)}_N := \{Y > \frac{N }{4}\}.$
Markov's inequality applied to  $\mathbbm{P}\left(\left(\sum_{i=1}^{N} (Y_i-1) \right)^4 \ge \frac{N^4}{4^4}\right)$ together with the assumption $\mathbb E[Y^4] < \infty$ implies  $\mathbbm{P}(E_N^{(1)})  = O(N^{-2})$. Likewise, $\mathbbm{P}(E_N^{(3)}) = O(N^{-4})$. Furthermore, let $K$ be so large that $\mathbb E[Y \wedge K] > \tfrac{3}{4}$. Again, Hoeffding's inequality applied to the sample mean of i.i.d.copies of the bounded random variable $Y \wedge K$ together with monotonicity imply that $\mathbb P(E_N^{(2)})$ decays exponentially; a fortiori we have $\mathbb P(E_N^{(2)})= O(N^{-3})$.\\
Let $E_N:= E_N^{(1)} \cup E_N^{(2)} \cup E_N^{(3)}$.
We have $\mathbbm{E}\left[ \frac{Y^2}{\left( \frac{H_N}{N} + \frac{Y}{N}\right)^2} \mathbbm{1}_{E_N} \right] = O(N^{-1})$ since on $E_N^{(1)}$ we have $\frac{Y^2}{\left( \frac{H_N}{N} + \frac{Y}{N} \right)^2} \leq Y^2$ and on  $E_N^{(2)}$ and $E_N^{(3)}$ we have
$\frac{Y^2}{\left( \frac{H_N}{N} + \frac{Y}{N}\right)^2} \leq N^2$. Hence, it remains to show that 
$\mathbbm{E}\left[ \frac{Y^2}{\left( \frac{H_N}{N} + \frac{Y}{N}\right)^2} \mathbbm{1}_{E_N^c} \right] = \rho^2 + O(N^{-1}).$ Define $Z_N := \frac{1}{\sqrt{N}} (H_N -N)$ and observe
\begin{align*}
  \mathbbm{E}\left[\frac{Y^2}{\left(\frac{H_N}{N} + \frac{Y}{N}\right)^2 } \mathbbm{1}_{E_N^c} \right] 
  = \mathbbm{E}\left[\frac{Y^2}{\left(1 + \frac{Z_N}{ \sqrt{N}} + \frac{Y}{ N}\right)^2 } \mathbbm{1}_{E_N^c} \right].
\end{align*} 
Abbreviate $R_N = 2 \left(\frac{Z_N}{ \sqrt{N}} + \frac{Y}{N}\right) +\left(\frac{Z_N}{ \sqrt{N}} + \frac{Y}{N}\right)^2 $.
On $E_N^c$ we have $-\frac{1}{2} \leq  R_N$ and hence
\begin{align*}
1 - R_N \leq  \frac{1}{\left(\frac{H_N}{N} + \frac{Y}{N}\right)^2}= \frac{1}{1 + R_N} \leq 1 - R_N + 2 R_N^2.
\end{align*}
Thus
\begin{align}\label{obenunten}
\mathbbm{E}[ Y^2 (1- R_N) \mathbbm{1}_{E_N^c}] \leq  \mathbbm{E}\left[ \frac{Y^2}{\left( \frac{H_N-N}{N} + \frac{Y}{N}\right)^2} \mathbbm{1}_{E_N^c} \right] & \leq \mathbbm{E}[Y^2 (1- R_N + 2 R_N^2) \mathbbm{1}_{E_N^c}]. 
\end{align}
By Cauchy-Schwarz we have $\mathbbm{E}[Y^2 \mathbbm{1}_{E_N^c}] = \mathbbm{E}[Y^2] + O(N^{-1})$. Similarly, $\mathbbm{E}[Y^2 Z_N  \mathbbm{1}_{E_N}] =  \mathbbm{E}[Y^2 Z_N] + O(N^{-1}) = O(N^{-1})$, since $\mathbbm{E}[Y^2 Z_N]$ vanishes due to the independence of $Y$ and $Z_N$. The remaining terms in \eqref{obenunten} are $O(N^{-1})$ as well, which completes the proof of Lemma \ref{LemmaSecondMoment}.
}
\end{proof}

\noindent We now prove a  bound on the deviations for the total weight of $k$ individuals.
\begin{lemma}(Large deviations bound for a moderate number of random weights) \label{Lemma Large Deviations Weights} \qquad \\
	Let $(Y_i)$ and $(W^{(N)}_i)$ satisfy \eqref{Dirichlettypeweights}, \eqref{Generating function Y}, $\mathbbm{E}[Y]=1$ and let $k = k_N \leq N^{c}$ for some $0<c \leq 1 $. Then for all $\eps> 0$ there exists a  positive constant $c_\eps$ depending only on $\eps$ and the distribution of $Y$ such that
	\begin{align}\label{Deviations random weights}
	\pp{\sum_{i=1}^{k} W_i^{(N)} \geq (1+\eps ) N^{c-1}} = O( \exp(-c_{\eps} N^{c} ) ).
	\end{align}	
\end{lemma}
\begin{proof} This follows by a combination of two Cram\'er bounds. Indeed, the l.h.s. of \eqref{Deviations random weights} is by assumption bounded from above by
	\begin{align*}
	\pp{\frac{\sum_{i=1}^{\lceil N^c \rceil } Y_i}{\sum_{j=1}^{N}Y_j} \geq (1+\eps ) N^{c-1}}. \label{intermediate1}
	\end{align*}
	\blue{ Abbreviating $E:=\{ \sum_{j=1}^{N} Y_j \geq (1-\eps') N\}$ with $\eps'$ such that $(1+\eps)(1-\eps')>1$ }we estimate the latter probability from above by
	\begin{align}
 &\pp{\sum_{i=1}^{\lceil N^c \rceil } Y_i \geq (1+\eps ) N^{c-1}  \sum_{j=1}^{N} Y_j , E}+\pp{E^c}\\
	&\blue{ \leq \pp{\sum_{i=1}^{\lceil N^c \rceil } Y_i \geq N^c (1+\eps )(1-\eps')}+\pp{E^c} }\\
	&\blue{  = O(e^{-N^c I((1+\eps)(1-\eps')  )} ) +O(e^{-N I(1-\eps'  )} ) },
	\end{align}
	 denoting by $I(y)$ the rate function of $Y$. Due to \eqref{Generating function Y} $I(y)$ exists around $\mathbb E[Y]=1$ and is strictly positive for $y \neq 1$ (see \cite{DZ} Theorem 2.2.3). This yields \blue{ an upper bound of $O(\exp(-c_\eps N^c))$ with $c_\eps =\min\{ I((1+\eps)(1-\eps')), I(1-\eps') \}$.}
\end{proof}
The next lemma gives stochastic upper and lower bounds for the sums of the random weights in terms of sums of the independent random variables $Y_i$.
\begin{lemma}(Bounds for the random weights)\label{Lemma Independent Weights} \qquad \\
	Assume that Conditions \eqref{Dirichlettypeweights} and \eqref{Generating function Y}  are fulfilled and $\mathbbm{E}[Y]=1$. Let $0<\alpha <\frac{1}{2}$, then for $k=k_N \leq~N$
	\begin{align}
	\pp{   \frac{1- N^{-\alpha} }{N } \sum_{i=1}^{k} Y_i \leq \sum_{i=1}^k W_i^{(N)} \leq  \frac{1+ N^{-\alpha} }{N } \sum_{i=1}^{k} Y_i  } \geq 1- \exp\left( - c' N^{1-2\alpha}\right) (1+o(1)),
	\end{align}
	for some $c'>0$.
\end{lemma}
\begin{proof}
	It suffices to show 
	\begin{align}
	\pp{ \left|   \frac{N }{\sum_{j=1}^{N}  Y_j}- 1 \right| \geq N^{-\alpha} }=O(\exp(\blue{-}N^{1-2\alpha})).
	\end{align}
	For $0<c< 1$ we have 
	\begin{align}
	\pp{ \sum_{i=1}^{N} Y_i < cN} = O( \exp(- N I(c))),
	\end{align}
	where $I(y)$ is the rate function of $Y$. Condition \eqref{Generating function Y} ensures that $I(c)>0$ for $\EE{Y_1} =1 \neq c$, see \cite{DZ} Theorem 2.2.3.\\ 
	For any $a,a' \geq 1$ one has $\left| \frac{1}{a}-\frac{1}{a'} \right|  \leq |a-a'|$. This yields
	\begin{align}
	\pp{ \left| \frac{N }{\sum_{j=1}^{N}  Y_j}- 1\right| \geq N^{-\alpha} }&= \pp{ \frac{1}{c} \left|   \frac{N c }{\sum_{j=1}^{N}  Y_j}- c \right| \geq N^{-\alpha} }\\
	&\leq \pp{ \frac{1}{c^2} \left|   \frac{\sum_{j=1}^{N}  Y_j}{N}- 1\right| \geq N^{-\alpha} } + O( e^{ - N I(c)} )\\
	&= \pp{ \frac{1}{\sqrt{N}} \left|  \sum_{i=1}^N (Y_i -1 ) \right| \geq c ^2 N^{\frac{1}{2}-\alpha}} +O( e^{ - N I(c)} ). \label{Use Cramer}
	\end{align}
Using \cite{Cramer1938} Theorem 1, the probability on the r.h.s. can, with a suitable $\widetilde{c} >0 $, be estimated from above by 
	\begin{align}
	\exp (\widetilde{c} N ^{1-3\alpha } ) \exp( - \frac{c^4}{2} N^{1-2\alpha}) (1+O(N^{-\alpha} \log N))= \exp\left(-\frac{c^4}{2} N^{1 - 2 \alpha}\right) (1+o(1)),
	\end{align}
	which gives the desired result.
\end{proof}

\section{Proof of  the main result}\label{SecProofofTheorem}
Recall from \eqref{DefbN} that we denote  the order of the selection strength by $b_N= \blue{-} \frac{\log s_N}{\log N}$. To simplify notation we will drop the subscript and simply write $b:=b_N$. As mentioned already in the sketch of the proof of Theorem \ref{HaldaneThm} we assume without loss of generality that $\mathbbm{E}[Y]=1.$

The proof of the Theorem is divided into three parts, which correspond to three phases of growth for the Cannings frequency process $\mathcal X$. The initial phase is decisive: due to  Proposition~\ref{Proposiotion reaching critical}, the probability that $\mathcal X$ reaches the level $N^{b+\delta}$ for some sufficiently small $\delta$ is given by the r.h.s. of \eqref{Haldane}. Lemma \ref{Lemma Hitting epsN} and Lemma \ref{Lemma hitting N} then guarantee that, once having reached the level $N^{b+\delta}$, the process $\mathcal X$ reaches $N$ with high probability. The proof of the Theorem is then a simple combination of these three results and the strong Markov property. Indeed, with $\tau_1,\tau_2,\tau_3$ as in Proposition \ref{Proposiotion reaching critical}, Lemma \ref{Lemma Hitting epsN} and Lemma \ref{Lemma hitting N}, and with $\delta, \delta',\eps$ fulfilling the requirements specified there,  the fixation probability in the l.h.s. of \eqref{Haldane} can be rewritten as
	\begin{align}
	\PP(X_{\tau}=N) &= \PP(X_{\tau_3}=N| X_{\tau_2} \geq \eps N ) \PP(X_{\tau_2} \geq \eps N| X_{\tau_1} \geq N^{b+\delta} ) \PP_1(X_{\tau_1}\geq N^{b+\delta}) \\
	&= (1-o(1)) (1-O(N^{-\delta'})) \frac{2 s_N}{\rho^2} (1+o(1))\\
	&\blue{\sim}\frac{2 s_N}{\rho^2}.
	\end{align}
	
\subsection{First phase: From $1$ to $N^{b+\delta}$}
In this section we show that as long as $X_g \leq N^{b+\delta}$ the process $\mathcal{X}$ can be upper and lower bounded (with sufficiently high probability) by two slightly supercritical branching processes $ \underline{ \mathcal{Z} } = (\underline{Z}_g)_{g\geq 0}$ and $\overline{ \mathcal{Z} }=(\overline{Z}_g)_{g\geq 0}$. 
To construct the upper bound $\overline{ \mathcal{Z}}$ we take the highest per capita selective advantage, which occurs when only a single individual is beneficial. Using Lemma \ref{Lemma Large Deviations Weights} and Lemma \ref{Lemma Independent Weights}, we will approximate the thus arising mixed binomial distribution by a mixed Poisson distribution, which leads for $\overline{ \mathcal{Z}}$ to the offspring distribution 
\begin{align}\label{offspringupper}
\text{Pois}\left( Y_1 (1+ s_N + o(s_N))\right),
\end{align}
where $Y_1$ is the random variable figuring in \eqref{Dirichlettypeweights}.
To arrive at the lower bounding Galton-Watson process $\underline{ \mathcal{Z} }$ we note that the per capita selective advantage is bounded from below by the one when $\lceil N^{b+\delta} \rceil$ beneficial individuals are present in the parent generation, as long as the process $\mathcal X$ has not reached the level $N^{b+\delta}$. Again using  Lemma \ref{Lemma Large Deviations Weights} and Lemma \ref{Lemma Independent Weights} we will show that the offspring distribution of $\underline{\mathcal Z}$ can be chosen as the mixed binomial distribution 
\begin{align}\label{offspringlower}
\text{ Bin} \left( N - \lceil N^{b+\delta} \rceil , \frac{Y_1}{N } (1+ s_N + o(s_N)) \right).
\end{align}

\begin{lemma} (Coupling with Galton-Watson processes)\label{Lemma Coupling with GWP} \qquad \\
   Let $\delta$ and $\alpha$ be such that $0<\delta< \eta$ and \blue{$\frac{1}{2}-\eta < \alpha < \frac{1}{2}$ }, and put $$\tau_1=\inf \{ g\geq 0 : X_g =0 \text{ or } X_g \geq N^{b+\delta}\}.$$ Then $\mathcal{X}$ can be defined on one and the same probability space together with two branching process $\underline{ \mathcal{Z}}$ and $\overline{\mathcal{Z}}$ with offspring distributions  \eqref{offspringlower} and \eqref{offspringupper}, respectively, such that for $j=1,2,\ldots$
	\begin{align}\label{inductive}
	\PP(\underbar{Z}_{j \wedge \tau_1} \blue{  \wedge \lceil N^{b+\delta} } \rceil \leq  X_{j\wedge \tau_1}& \blue{ \wedge \lceil N^{b+\delta} \rceil }\leq \overline{ Z}_{j\wedge \tau_1}\, \big |\, \underbar{Z}_{j-1\wedge \tau_1} \leq X_{j-1\wedge \tau_1} \leq \overline{ Z}_{j-1\wedge \tau_1})\\
	&\geq 1- e^{ -c' N^{1-2\alpha}} (1+o(1)),
	\end{align}
	with $c'$ as in Lemma \ref{Lemma Independent Weights}.
\end{lemma}
\noindent Applying the latter estimate $g$ times consecutively yields immediately the following corollary:
\begin{corollary}  \label{Cor coupling}
Let  $\delta, \alpha, \tau_1,\underline{\mathcal{Z}}$ and $\overline{\mathcal{Z}}$ be as in Lemma \ref{Lemma Coupling with GWP}. If $X_0 \leq N^{b+\delta}$, then for all $g \in \N_0$
	\begin{align}\label{CouplingGW}
	\PP(\underbar{Z}_{g \wedge \tau_1}\blue{  \wedge \lceil N^{b+\delta} \rceil } \leq X_{g\wedge \tau_1} \blue{ \wedge \lceil N^{b+\delta} \rceil } \leq \overline{ Z}_{g\wedge \tau_1}|\underbar{Z}_{0} \leq X_{0} \leq \overline{ Z}_{0}) \geq \left( 1- O(\exp(-c'N^{1-2\alpha} )) \right)^g.
	\end{align}
\end{corollary}

\begin{proof}\emph{of Lemma \ref{Lemma Coupling with GWP}}.
	We proceed inductively, assuming that for $g=1,2,\ldots$ we have constructed $\mathcal X$, $\overline {\mathcal Z}$ and $\underline {\mathcal Z}$ up to generation $g-1$ such that \eqref{inductive} holds for $j=1, \ldots, g-1$. Together with $X_g$ we will construct $\overline Z_g$ and $\underline Z_g$, and check the asserted probability bound for the coupling.\\
	\noindent 
	Given $\{X_{g-1}=k\}$ and
	the weights~$(W_i)$ in generation $g-1$, the number of beneficial individuals $X_g$ in generation $g$ has the binomial distribution 
	\eqref{Transition probabilities beneficial}. 
	Aiming first at the construction of the upper bound  $\overline {\mathcal Z}$, we relate \eqref{Transition probabilities beneficial}	to \eqref{offspringupper} in terms of stochastic order.
	For $p, p'\geq 0$, a \text{Bin}$(N,p)$-distributed random variable $B$ is stochastically dominated by a  Pois$(Np')$-distributed random variable $P$ if
	\begin{align}
	e^{- p'} \leq (1-p), \label{Condition stochastic domination}
	\end{align}
	 see (1.21) in \cite{Klenke2010}.
Indeed, in this case the probability of the outcome zero is not larger for a Pois($p'$)-distributed random variable $P_1$ than for a Bernoulli($p$)-distributed random variable $B_1$, which yields $B_1 \preceq P_1$, where  $\preceq$ denotes the usual stochastic ordering of the random variables. Consequently
	\begin{align}
		B\stackrel{d}{=} \sum_{i=1}^{N} B_i \preceq \sum_{i=1}^N P_i \stackrel{d}{=} P.
	\end{align}
with $B_i$ and $P_i$ being independent copies of $B_1$ and $P_1$, respectively. 
    In particular, for $p\geq 0$ and $p'= p (1+N^{b +2 \delta -1}) $ we have
	\begin{align}
	e^{-p'} \le 1-p'+ (p')^2 = 1-p (1+N^{b +2 \delta -1}) + p^2 (1+N^{b +2 \delta -1})^2.
	\end{align}
	Hence Condition \eqref{Condition stochastic domination} holds if
	\begin{align}
	p (1+N^{b +2 \delta -1})^2&< N^{b +2 \delta -1}. \label{Condition on p}
	\end{align}Given $(W_i)$, the success probability of the binomial distribution \eqref{Transition probabilities beneficial} is bounded from above via   $$p:=\left(\sum_{i=1} ^{k} W_i\right)/(1-s_N).$$
	Thus by Lemma \ref{Lemma Large Deviations Weights}, \eqref{Condition on p} is fulfilled with probability  $1- O(\exp(-c_\eps N^{b+\delta}))$ with $c_\eps$ as in Lemma \ref{Lemma Large Deviations Weights}. In this sense the number of beneficial offspring is dominated by a Pois$\left( N \frac{\sum_{i=1} ^{k} W_i}{(1-s_N)} (1+N^{b+2\delta-1}) \right) $-distributed random variable with high probability. Applying Lemma \ref{Lemma Independent Weights} yields that with probability $1-\exp(-c' N^{1-2\alpha})(1+o(1))$ the following chain of inequalities is valid:
	\begin{align}
	N \frac{\sum_{i=1} ^{X_{g-1}} W_i}{(1-s_N)}(1+N^{b+2\delta-1}) &\leq \frac{\sum_{i=1} ^{X_{g-1}} Y_i}{(1-s_N)} (1+N^{b+2\delta-1}) (1+N^{-\alpha}) \\
	&=\sum_{i=1}^{X_{g-1}} Y_i (1+s_N+o(s_N)) \leq \sum_{i=1}^{\overline{Z}_{g-1}} Y_i (1+s_N+o(s_N)).
	\end{align}
	In this way $\mathcal{X}$ can be coupled with a branching process $\overline{ \mathcal{Z} }$ with a mixed Poisson offspring distribution of the form \eqref{offspringupper}.\\\\
	The lower bound also uses a comparison with a Galton-Watson process, now with a mixed binomially distributed offspring distribution:

Number the individuals in generation $g-1$ by $(i,g-1)$,  with $(1,g-1), \ldots, (X_{g-1},g-1)$,  being the beneficial individuals. Given $\mathscr W$, we use a sequence of coin tossings to determine which of the individuals from generation $g$ are the children of $(i,g-1)$. The first $N$ tosses determine which individuals are the children of $(1,g-1)$. Denoting the number of these children by $\omega_1^{g-1}$, the next $N-\omega_1^{g-1}$ tosses (with an updated success probability) determine  which individuals are the children of $(2,g-1)$, etc.    Observe that as long as $X_{g-1} \leq N^{b+\delta}$, and given $\mathscr W$ and $\sum_{\ell=1}^{i-1}\omega_\ell^{(g-1)}=:h$, then $\omega_i^{(g-1)}$ for $i \leq X_{g-1}$ has distribution
\begin{align}\label{Binrec}
\text{ Bin} \left(N-h , \frac{W_i}{\sum_{\ell = i}^{X_{g-1}}W_\ell + (1-s_N) \sum_{\ell= X_{g-1} +1}^{N}  W_\ell} \right).
\end{align}
Note that the success probability in \eqref{Binrec} can be estimated from below by
	\begin{align}
	\frac{W_i}{\sum_{\ell=1}^{X_{g-1} } W_\ell + (1-s_N) \sum_{\ell= X_{g-1} +1 }^N W_\ell } = \frac{W_i}{1-s_N +s_N \sum_{\ell=1}^{X_{g-1}} W_\ell}. \label{intermediate2}
	\end{align}
	As long as $X_{g-1} \leq \lceil N^{b+\delta} \rceil$, Lemma \ref{Lemma Large Deviations Weights} ensures that  for $\eps > 0$	
	\begin{align}\label{intermediate3}
	\blue{\frac{W_i}{1-s_N +s_N \sum_{j=1}^{X_{g-1}} W_j}} \geq \frac{W_i}{1-s_N + (1+\eps)N^{\delta -1}}
	\end{align}
	with probability $1-O(\exp(-c_\eps N^{b+\delta}))$.  Lemma \ref{Lemma Independent Weights}, in turn,  yields that the r.h.s. of \eqref{intermediate3} is bounded from below by 
		\begin{align}
	 \frac{Y_i}{N (1-s_N + (1+\eps)N^{\delta -1})} (1+ N^{-\alpha}) 
= \frac{Y_i}{N } (1+ s_N +o(s_N))
	\end{align}
		 with probability at least $1-\exp(-c' N^{1-2\alpha})(1+o(1))$. 
		 
		 Thus, 
	if $\omega_1^{(g-1)}+\dots+\omega_{i-1}^{(g-1)}=h \leq \blue{ \lceil N^{b+\delta} \rceil}$, then the distribution of $\omega_i^{(g-1)}$  specified in \eqref{Binrec} is bounded from below by
		\begin{align}
	\text{Bin}\left( N-\lceil N^{b+\delta} \rceil , \frac{W_i}{1-s_N + (1+\eps)N^{\delta -1}} \right)
	\end{align}
	with probability $1-O(\exp(-c_\eps N^{b+\delta}))$.

\blue{If $\omega_1^{(g-1)}+\dots+\omega_{i-1}^{(g-1)}=h > \lceil N^{b+\delta} \rceil $, then we have $\underbar{Z}_{g \wedge \tau_1}  \wedge \lceil N^{b+\delta} \rceil  \leq X_{g\wedge \tau_1} \wedge \lceil N^{b+\delta} \rceil $}.
	Consequently $\mathcal{X}$ can be coupled with a Galton-Watson process $\underline{\mathcal{Z}}$ with offspring distribution of the form \eqref{offspringlower} such that also the lower estimate in \eqref{CouplingGW} is fulfilled. This completes the proof of Lemma \ref{Lemma Coupling with GWP}.
\end{proof}
  
We are now ready to prove that $\mathcal{X}$ reaches the level $N^{b+\delta}$ with probability $\frac{2s_N}{\rho^2} (1+o(1))$.
\begin{proposition}(Probability to reach the critical level) \label{Proposiotion reaching critical} \qquad \\
	Assume Conditions \eqref{Dirichlettypeweights}, \eqref{Generating function Y} and \eqref{Condition s_N} are fulfilled and define $\tau_1 = \inf \{ g \geq 0 : X_g \geq N^{b+\delta} \text{ or } X_g=0 \}$ with $0<\delta < \eta$, then
	\begin{align}
	\PP( X_{\tau_1} \geq N^{b+\delta} ) = \frac{2 s_N}{\rho^2 }(1+o(1)). \label{prob critical level}
	\end{align}
\end{proposition}
\begin{proof}
	We use the couplings  of $\mathcal{X}$ with the slightly supercritical branching processes $\underline{ \mathcal{Z} }$ and $\overline{ \mathcal{Z} }$ from Corollary \ref{Cor coupling} and show that both processes reach the level $N^{b+\delta}$ with probability $\frac{2s_N}{\rho^2}(1 + o(1))$.
	Let \blue{$\delta'>0$} and $E$ be the event that the stochastic ordering between $\underline{\mathcal{Z}},\mathcal{X}$ and $\overline{ \mathcal{Z}}$ holds until generation $n_0=\lceil N^{b+\delta'} \rceil$, that is
	\begin{align}
	E= \{\underline{Z}_0  \blue{ \wedge \lceil N^{b+\delta} \rceil } \leq X_0  \blue{ \wedge \lceil N^{b+\delta} \rceil } \leq \overline{ Z}_0,...,\underline{Z}_{n_0} \blue{ \wedge \lceil N^{b+\delta} \rceil } \leq X_{n_0}  \blue{ \wedge \lceil N^{b+\delta} \rceil } \leq \overline{ Z}_{n_0} \}.
	\end{align}
	We show below that the stopping time $\tau_1$ fulfils
	\begin{align}
	\PP(\tau_1 \geq \lceil N^{b+\delta'} \rceil ) = o(s_N). \label{Bound tau_1}
	\end{align}
	For some $g$ that is polynomially bounded in $N$, the r.h.s. of \eqref{CouplingGW} is bounded from above by $1-o(s_N)$. Thus, 
	combining Corollary \ref{Cor coupling} and $\eqref{Bound tau_1}$ we deduce
	\begin{align}
	\PP(E, \tau_1 \leq \lceil N^{b+\delta'} \rceil ) =1-o(s_N).
	\end{align}
	We are now going to bound \eqref{prob critical level} from above by estimating the corresponding probability for $\overline {\mathcal Z}$ and the stopping time $\overline{\tau}_1 = \inf \{ g \geq 0 : \overline{ Z}_g \geq N^{b+\delta} \text{ or } \overline{Z}_g=0 \}$. More precisely,
	\begin{align}
	\PP( X_{\tau_1} \geq N^{b+\delta} )&= \PP( X_{\tau_1} \geq N^{b+\delta}, \tau_1 \leq \lceil N^{b+\delta'} \rceil ,E )+ o(s_N) \\
	&\leq \PP( \overline{Z}_{\overline{\tau}_1} \geq N^{b+\delta}, \tau_1 \leq \lceil N^{b+\delta'} \rceil ,E )+ o(s_N)\\
	&\leq \PP( \overline{Z}_{\overline{\tau}_1} \geq N^{b+\delta})+ o(s_N).
	\end{align}	
To obtain an upper bound for the probability of $\overline{\mathcal{Z} }$ to reach the level $N^{b+\delta}$  it suffices to estimate the survival probability of $\overline{\mathcal{Z} }$. For notational simplicity let us write $\{ \overline{ \mathcal{Z}} \text{ survives} \} $ for the event $\{ \forall g\geq 0: \overline{Z}_g >0 \}$ and similarly $\{ \overline{ \mathcal{Z}} \text{ dies out} \}$  for the event $\{ \exists g \geq 0 : \overline{Z}_g =0 \}$. We have
	\begin{align}
	&\PP_1 \left( \overline{Z}_{\overline{\tau}_1} \geq N^{b+\delta} \right)\\
	\leq\, &\PP_1 \left(  \overline{Z}_{\overline{\tau}_1} \geq N^{b+ \delta}|\overline{\mathcal{Z}} \text{ survives }\right) \PP_1( \overline{\mathcal{Z}} \text{ survives }) +\PP_1 \left(  \overline{Z}_{\overline{\tau}_1} \geq N^{b+\delta}|\overline{\mathcal{Z}} \text{ dies out }\right)\\
	\leq\, &  \PP_1( \overline{\mathcal{Z}} \text{ survives }) + \pp{\text{all } \lceil N^{b + \delta} \rceil  \text{ individuals die out}} =\PP_1( \overline{\mathcal{Z}} \text{ survives }) + (1-\blue{\PP_1( \overline{\mathcal{Z}} \text{ survives })}) ^{ \lceil N^{b+ \delta} \rceil }.
	\end{align}

The survival probability of $\overline{\mathcal{Z}} $ will now be estimated by means of \eqref{equaiton surv prob GW}. To this purpose we calculate the expectation and variance of the offspring distribution \eqref{offspringupper}. \\ The expectation is $1+s_N +o(s_N)$ and the variance is given by
	\begin{align}
	&\VV{  \text{Pois}\left( Y_1 (1+s_N+o(s_N) ) \right)}\\
	&=\VV{\EE{ \text{Pois}\left( Y_1(1+s_N+o(s_N) )  \right)\Big| Y_1 } } +\EE{\VV{\text{Pois} \left( Y_1(1+s_N+o(s_N) )  \right)\Big| Y_1  }}\\
	&=\VV{ Y_1 (1+s_N+o(s_N) ) } +\EE{Y_1(1+s_N+o(s_N) )}\\
	&= (1+s_N+o(s_N))^2 \VV{Y_1} +1+s_N+o(s_N)= \rho^2 (1+o(1)).
	\end{align}
Equation \eqref{equaiton surv prob GW} yields that the survival probability of the process $\overline{\mathcal{Z}}$ is given by $\frac{2 s_N}{\rho ^2}(1+o(1))$. The lower bound in \eqref{prob critical level} follows by similar arguments by considering the process $\underline{\mathcal{Z}}$ instead.\\
    
	It remains to show \eqref{Bound tau_1}.
	Define $\overline{\tau}^{(0)}$ and $ \underline{\tau}^{(\text{u})}$ as the stopping times that the process $\overline{\mathcal{Z}}$ reaches~$0$ and the process $ \underline{\mathcal{Z}}$ reaches the upper bound $N^{b+\delta}$, respectively, i.e.
	\begin{align}
	\overline{\tau}^{(0)} = \inf \{ g \geq 0 : \overline{ Z}_g =0 \},  \qquad 
		\underline{\tau}^{(\text{u})}=\inf \{ g \geq 0 : \underline{Z}_g \geq N^{b+\delta} \},
	\end{align}
	with the convention that the infimum over an empty set is infinity. 
	Then
	\begin{align}
	\PP(\tau_1 \geq \lceil N^{b+\delta'} \rceil )&\leq \PP(\tau_1 \geq \lceil N^{b+\delta'} \rceil, E )+ \PP(E^c ) \\
	&\leq \PP(\underline{\tau}^{(\text{u})} \geq \lceil N^{b+\delta'} \rceil, \overline{\tau}^{(0)} \geq \lceil N^{b+\delta'} \rceil ,E)+ \PP(E^c ) 
\\
	&= \PP(\underline{\tau}^{(\text{u})} \geq \lceil N^{b+\delta'} \rceil, \overline{\tau}^{(0)} \geq \lceil N^{b+\delta'} \rceil ,E,\underline{ \mathcal{Z} } \text{ dies out } ) \\
	&\qquad +\PP(\underline{\tau}^{(\text{u})} \geq \lceil N^{b+\delta'} \rceil, \overline{\tau}^{(0)} \geq \lceil N^{b+\delta'} \rceil ,E,\underline{ \mathcal{Z} } \text{ survives } ) + \PP(E^c ) \\
	&=\PP(\underline{\tau}^{(\text{u})} \geq \lceil N^{b+\delta'} \rceil, \overline{\tau}^{(0)} \geq \lceil N^{b+\delta'} \rceil ,E,\underline{ \mathcal{Z} } \text{ dies out } )\\
	&\qquad +O(e^{-c' N^{\frac{\delta'}{2}} }) + O(N^{b+\delta'} e^{- \frac{1}{2}N^{1-2\alpha})}, \label{Proof stopping time 2}
	\end{align}
	by an application of Lemma \ref{Lemma ext time GWP} and Corollary \ref{Cor coupling} and $\alpha< \frac{1}{2}$ as defined there. \blue{To keep the notation simple} we denote by $e_N$ terms of the order $\exp({-N^{c}})$ for some $c>0.$  Proceeding with \eqref{Proof stopping time 2} we obtain
	\begin{align}
	\eqref{Proof stopping time 2}&= \PP(\underline{\tau}^{(\text{u})} \geq \lceil N^{b+\delta'} \rceil, \overline{\tau}^{(0)} \geq \lceil N^{b+\delta'} \rceil ,E,\underline{ \mathcal{Z} } \text{ dies out }, \overline{ \mathcal{Z} } \text{ survives} )\\
	&\qquad +\PP(\underline{\tau}^{(\text{u})} \geq \lceil N^{b+\delta'} \rceil, \overline{\tau}^{(0)} \geq \lceil N^{b+\delta'} \rceil ,E,\underline{ \mathcal{Z} } \text{ dies out }, \overline{ \mathcal{Z} } \text{ dies out} ) +e_N \\
	&\leq \PP(\underline{ \mathcal{Z} } \text{ dies out }, \overline{ \mathcal{Z} } \text{ survives}, E ) + e_N, \label{Proof stopping time 3}
	\end{align}
again by an application of Lemma \ref{Lemma ext time GWP}. Note that
	\begin{align}
	&\PP( \underline{ \mathcal{Z} } \text{ dies out }, \overline{ \mathcal{Z} } \text{ survives}, E ) + \PP( \underline{ \mathcal{Z} } \text{ survives }, \overline{ \mathcal{Z} } \text{ survives}, E )\\
	= &\PP( \overline{ \mathcal{Z} } \text{ survives}, E) = \frac{2s_N}{\rho^2}(1+o(1)).
	\end{align}
	In order to show that \eqref{Proof stopping time 3} is $o(s_N)$ it suffices to prove that 
	\begin{align}
	\PP( \underline{ \mathcal{Z} } \text{ survives }, \overline{ \mathcal{Z} } \text{ survives}, E )= \frac{2s_N}{\rho^2} (1+o(1)). \label{Proof stopping time final}
	\end{align}
Considering again the event $\{ \underline{\tau}^{(u)} \leq \lceil N^{b+\delta'} \rceil \}$ and applying \eqref{equaiton surv prob GW} one obtains
	\begin{align}
		\PP( \underline{ \mathcal{Z} } \text{ survives }, \overline{ \mathcal{Z} } \text{ survives}, E )&= \PP( \underline{ \mathcal{Z} } \text{ survives }, \overline{ \mathcal{Z} } \text{ survives}, E , \underline{\tau}^{(u)} \leq \lceil N^{b+\delta'} \rceil ) + e_N \\
		&=\PP( \underline{ \mathcal{Z} } \text{ survives },E , \underline{\tau}^{(u)} \leq \lceil N^{b+\delta'} \rceil ) + e_N,
	\end{align}
since the events $E$ and $\{  \underline{\tau}^{(u)} \leq \lceil N^{b+\delta'} \rceil )\}$ imply that $\overline{ Z}_g \geq N^{b+\delta}$ for some $g \leq N^{b+\delta'}$ and the probability for $\overline{ \mathcal{Z} }$ to die out after reaching $N^{b+\delta}$ is $(1-\frac{2s_N}{\rho^2}(1+o(1)))^{N^{b+\delta}}=e_N $.
One more application of \eqref{equaiton surv prob GW} yields
	\begin{align}
	\PP( \underline{ \mathcal{Z} } \text{ survives },E , \underline{\tau}^{(u)} \leq \lceil N^{b+\delta'} \rceil )= \frac{2s_N}{\rho^2}(1+o(1)),
	\end{align}
which finishes the proof. 
\end{proof}
\subsection{Second phase: from $N^{b+\delta}$ to $\eps N$} \label{Sec Second Phase}
 In this section we show that $\mathcal{X}$, once having reached the level $N^{b+\delta}$, will reach the level $\eps N$ with probability tending to $1$ as $N\to \infty$. 
\begin{lemma}(From $N^{b+\delta}$ to  $\varepsilon N$  with high probability) \label{Lemma Hitting epsN} \qquad \\
	Assume $X_0 \geq N^{b+\delta}$ with $0<\delta<\eta$, let $0<\eps<\frac{\delta}{2-2\eta-\delta}$ and 
	define the stopping time \begin{align*}\tau_2 =\inf \{g \geq 0 : X_g \notin \{1,2,..., \lfloor \eps N \rfloor \}\}. 
	                          \end{align*}
Then there exists some $\delta'>0$ such that
	\begin{align}
	\pp{X_{\tau_2} \geq \eps N} = 1-O(N^{-\delta'}). \label{equation Lemma espN}
	\end{align}
\end{lemma}
	\begin{proof}
By monotonicity it is enough to prove the claim for $X_0 = \lceil N^{b+\delta} \rceil$.
By definition we have
	\begin{align} \label{transprobX}
	\mathscr L (X_{g+1}|X_g)=\text{Bin}\left( N,\frac{\sum_{i=1}^{X_g} W_i }{\sum_{i=1}^{X_g} W_i+(1-s_N) \sum_{i=X_g +1}^{N} W_i} \right).
	\end{align}
Next we lower-bound $\mathcal{X}$ by the process $\widetilde{\mathcal{X}} = (\widetilde{X}_g)_{g\geq 0}$, $ \widetilde{X}_0 = X_0$, with conditional distribution 
	\begin{align}
	\mathscr L (\widetilde{X}_{g+1}|\widetilde{X}_g)=\text{Bin}\left( N,\frac{\sum_{i=1}^{\widetilde{X}_g} W_i}{ 1-s_N \sum_{i=\eps N +1}^{N} W_i} \right) \label{distribution Xtilde}
	\end{align}
as long as $\widetilde{X}_g \leq \eps N$. If $\widetilde{X}_g > \eps N$ we assume that $\widetilde{X}_{g+1}$ is distributed as a slightly supercritical branching process with Pois$(Y_1 q_N)$ distributed offspring,  where 
\begin{align}\label{defq}
q_N = N \mathbb E\left[\frac{W_1}{1-s_N\sum_{i=\eps N+1}^N W_i}\right].
\end{align}
We will see that by this definition in each generation the expectation of $\widetilde {\mathcal X}$ increases  
by the factor $q_N$, see \eqref{variance and exp} and \eqref{expb}. The generation-wise increase of the variance conditioned on the current state can be estimated from above by a factor $\rho^2(1+ o(1))$, see \eqref{variance and exp} and \eqref{variance estimate}, leading to an iterative estimate on the variance of the form \eqref{varianceEst}. 
As long as $X_g \geq \widetilde{X}_g$, the success probability in the mixed Binomial distribution on the r.h.s. of \eqref{transprobX} dominates the corresponding one on the r.h.s. of \eqref{distribution Xtilde}.
	Consequently, starting $\widetilde{X}$ and $X$ both in $\lfloor N^{b+\delta} \rfloor$ we can couple them, such that $\widetilde{X}_g \leq X_g$ as long as $\widetilde{X}$ did not cross the level $\eps N$.
	 In particular, we have for $\widetilde{\tau}=\inf \{ g \geq 0 : \widetilde{X}_g \notin \{1,2,...,\eps  N\} \}$
	\begin{align}
	\mathbb P\big(X_{\tau} \geq \eps N\big) \geq \mathbb P\big(\widetilde{X}_{\widetilde{\tau}}  \geq \eps N\big). 
	\label{Hitting Probability}
	\end{align}
	To show $\pp{\widetilde{X}_{\widetilde{\tau}} \geq \eps N} =  1-O(N^{-\delta'})$  we will estimate the first and second moment of $\widetilde{X}_{g_0}$ for a suitably chosen $g_0 \in \N$ and then use Chebyshev's inequality to show that $\widetilde{X}_{g_0}$ is above $\eps N$ with sufficiently high probability. For this purpose we consider $m(x)$ and $v(x)$, the one-step conditional expectation and variance of $\widetilde{X}$ at $x \in \N$, that is
	\begin{align}
	m(x)=\EE{\widetilde{X}_1|\widetilde X_0=x}, \qquad v(x)=\VV{\widetilde{X}_1|\widetilde X_0=x}.
	\end{align}
	From the definition of $\widetilde{\mathcal X}$ as a branching process above $\eps N$ we have for $x >\eps N$
	\begin{align}
		m(x)=q_N x, \qquad v(x)= \rho^2 x (1+o(1)). \label{variance and exp}
	\end{align}
	Next we show that $m(x)$ and $v(x)$ fulfil relations similar to \eqref{variance and exp} also for $x \leq \eps N$, which will allow to estimate the expectation and the variance of ${\widetilde{X}}_{g_0}$.
	
For $x \leq \eps N$ we have due to \eqref{distribution Xtilde}
	\begin{align}
	m(x) &= N \EE{\frac{\sum_{i=1}^{x} W_i}{ 1-s_N \sum_{i=\eps N +1}^{N} W_i}}\\
	&= N \EE{\sum_{i=1}^{x} W_i \left(1+s_N \sum_{i=\eps N +1}^{N} W_i+ O(s_N^2) \right) }\\
	&=x( 1+s_N (1-\eps) N^2 \EE{W_1 W_{\eps N +1} } + O(s_N^2))\\
	&=\blue{x\left ( 1+s_N (1-\eps)(1+O(N^{-1})) + O(s_N^2)\right) }\\
	&= x \left(1+(1-\eps ) s_N + O(s_N^2)\right).
	\end{align}
	\blue{ In the penultimate equality we used $\EE{W_1W_2}= \frac{1}{N^2} + O(N^{-3})$ which results from the fact
	\begin{align}
		1=\EE{\left(\sum_{i=1}^{N}W_i\right)^2 }=N \EE{W_1^2}+ N(N-1)\EE{W_1W_2}
	\end{align}	
	and \eqref{second moment W}}. Consequently, we have for all $x\in \mathbb N$, recalling \eqref{defq}, 
	\begin{align}\label{expb}
	m(x)= x q_N = x \left(1+(1-\eps ) s_N + O(s_N^2)\right).
	\end{align}

Next we analyse $v(x)$, again for $x \leq \eps N$. In view of \eqref{distribution Xtilde}, a decomposition of the variance gives
\begin{align}
v(x)&=\VV{N \frac{\sum_{i=1}^{x} W_i}{1-s_N \sum_{i=\eps N +1}^N W_i} } + \EE{ N \frac{\sum_{i=1}^{x} W_i}{1-s_N \sum_{i=\eps N +1}^N W_i} \left( 1- \frac{\sum_{i=1}^{x} W_i}{1-s_N \sum_{i=\eps N +1}^N W_i}\right)} \\
&\leq \blue{  \EE{ \left( N \frac{\sum_{i=1}^{x} W_i}{1-s_N \sum_{i=\eps N +1}^N W_i}\right)^2} - \EE{N \frac{\sum_{i=1}^{x} W_i}{1-s_N \sum_{i=\eps N +1}^N W_i}}^2 + \EE{ N \frac{\sum_{i=1}^{x} W_i}{1-s_N}} } \\
&\blue{ \leq \EE{ \left( N \frac{\sum_{i=1}^{x} W_i}{1-s_N}\right)^2}- \EE{N\sum_{i=1}^{x} W_i}^2 + \EE{ N \frac{\sum_{i=1}^{x} W_i}{1-s_N} }. }
\end{align}
Because of the negative correlation of the $W_i$, the sum of the first and the second term is not larger than $x N^2 \VV{W_1}$, which because of \eqref{second moment W} is $\le x (\rho^2 -1) + O(N^{-1})$. Since the third term is $x(1 + O(s_N))$, we  have for all $x \le \eps N$
\begin{align}
v(x) \leq \rho^2 x (1+o(1)). \label{variance estimate}
\end{align}
Combining \eqref{expb} and \eqref{variance estimate} allows us to estimate the variance $\VV{\widetilde{X}_g}$ for $g\in \N$, again by decomposing the variance:
\begin{align}\label{varianceEst}
\VV{\widetilde{X}_g}&= \VV{\EE{\widetilde{X}_g|\widetilde{X}_{g-1}}}+\EE{\VV{\widetilde{X}_g|\widetilde{X}_{g-1}}}\\
&= \VV{m(\widetilde{X}_{g-1})} + \EE{v(\widetilde{X}_{g-1})} \\
&\leq q_N^2 \VV{\widetilde{X}_{g-1}}+ \rho^2 \EE{\widetilde{X}_{g-1}} (1+o(1)) \\
&=q_N^2 \VV{\widetilde{X}_{g-1}}+ \rho^2 q_N^{g-1} \widetilde{X}_0 (1+o(1)).
\end{align}
Iterating this argument yields
\begin{align}
	\VV{\widetilde{X}_{g}}&=\rho^2 \widetilde{X}_0 q_N^{g-1} \sum_{j=0}^{g-1} q_N^j (1+o(1)) \\
	&= \rho^2 \widetilde{X}_0 q_N^{g-1} \frac{q_N^g-1}{q_N-1} (1+o(1)).
\end{align}

Choose the minimal $g_0 \in \N $ such that $2 \eps N \leq \EE{\widetilde{X}_{g_0}} = q_N^{g_0} X_0$, which yields recalling the initial condition $X_0 = \lceil N^{b+\delta}\rceil$ 
\begin{align*}
g_0= \left\lceil  \frac{\log (2 \eps N X_0^{-1})}{\log q_N} \right\rceil = \left\lceil  \frac{\log (2 \eps N^{1-b-\delta})}{(1-\eps) s_N + O(s_N^2)} \right\rceil.
\end{align*}
 Applying Chebyshev's inequality with $\widetilde{X}_0 = X_0$, we obtain
\begin{align}
\PP \left( |\widetilde{X}_{g_0} - \EE{\widetilde{X}_{g_0} }| \geq \eps N \right) &\leq \frac{\rho^2 \widetilde{X}_0 q_N^{g_0-1} \frac{q_N^{g_0}-1}{q_N-1} (1+o(1)) }{\eps^2 N^2}\\ \\
&\leq \frac{\rho^2 N^{b+\delta} q_N^{2g_0} \frac{N^b}{(1-\eps)} (1+o(1))}{\eps^2 N^2}\\
&= \frac{\rho^2}{\eps^2 (1-\eps)} N^{2b+\delta-2} (1+(1-\eps) s_N + O(s_N^2))^{2g_0} (1+o(1)) \\
&\leq c_{\rho,\eps} N^{2b+\delta-2} \exp( 2g_0 s_N (1+O(s_N))) (1+o(1))\\
&\leq c_{\rho,\eps} N^{2b+\delta-2} N^{\frac{2}{1-\eps} (1-b-\delta)} (1+o(1)) \\
&= O(N^{-\delta'}),
\end{align}
for some small $\delta'>0$ due to the assumptions on $\eps$.

Since $\EE{\widetilde{X}_{g_0}}\geq 2\eps N$, this implies
\begin{align}
\pp{\widetilde{X}_{\widetilde{\tau}}  \geq \eps N}\geq \PP(\widetilde{X}_{g_0} \geq \eps N ) \geq 1- O(N^{-\delta'})
\end{align}
and due to \eqref{Hitting Probability} this finishes the proof.
\end{proof}

\subsection{Third phase: from $\eps N$ to $N$} \label{Sec Third Phase}

Lemma \ref{Lemma hitting N} below concerns the last step of the proof, showing that once the process $\mathcal X$ has reached the level $\floor{\eps N}$, it goes to fixation with high probability. Our proof relies on a representation of the fixation probability  of $\mathcal X$ in terms of (a functional of) the equilibrium state $A_{\rm eq}:=A_{\rm eq}^{(N)} $ of the counting process $\mathcal A:= \mathcal{A}^{(N)}=(A_m)_{ m \geq 0}$ of the {\em potential ancestors} in  
 the time discrete Cannings ancestral selection graph as provided by \cite{BoeGoPoWa1}.  The process $\mathcal A^{(N)}$ is called {\em Cannings ancestral selection process (CASP)} in \cite{BoeGoPoWa1}; for fixed $N$, it is a recurrent, $[N]$-valued Markov chain whose transition probabilities are specified in \cite{BoeGoPoWa1} Section 2.3.  
 
 Theorem 3.1 and Formula (3.2) (see also Corollary 3.3) in \cite{BoeGoPoWa1}  provide the following sampling duality representation of the fixation probability of $\mathcal X$ when started with $k$ individuals:
 \begin{align} \label{duality}
 \mathbb P_k(\mathcal X \mbox{ eventually hits } N) = 1-\mathbb E\left[ \frac{(N-k)(N-k-1)\cdots (N-k-A_{\rm eq}+1)}{N(N-1)\cdots (N-A_{\rm eq}+1)}\right].
 \end{align}
Intuitively, this says that $\mathcal X$ goes extinct if and only if a random sample of (random) size $A_{\rm eq}$, drawn without replacement from the population of size $N$, avoids the $k$ beneficial individuals. 

Formula \eqref{duality} implies
\begin{align} \label{fixhighprob}
 \mathbb P_{\ceil{\eps N}}(\mathcal X \mbox{ eventually hits } N) \ge 1-\mathbb E \left[(1-\eps)^{A_{\rm eq}^{(N)}}\right].
 \end{align}
 The representation of the transition probabilities of $\mathcal A$ in \cite{BoeGoPoWa1} Section 2.3 in terms of two half steps yields that for fixed $N$ CASPs with different selection parameters can be coupled in such a way that $A_{\rm eq}^{(N)}$ is increasing in $s_N$. Take a sequence $(\tilde s_N)$ satisfying $\tilde s_N \le s_N$ and Condition (1.2) in \cite{BoeGoPoWa1}, i.e.
 \begin{align}
 N^{-1+\eta} \le \tilde s_N \le  N^{-2/3+\eta}.
 \end{align}
Let $\tilde A_{\rm eq}^{(N)}$ be the equilibrium state belonging to $\tilde s_N$ (and to the same Dirichlet-type paintbox as that of $\mathcal X$). The central limit result \cite{BoeGoPoWa1}, Corollary 6.10, implies that $\tilde A_{\rm eq}^{(N)}\to \infty$ in probability as $N\to \infty$. Because of the just mentioned monotonicity in the selection coefficient, the same convergence holds true for the sequence $\left(A_{\rm eq}^{(N)}\right)$. The following lemma is thus immediate from \eqref{fixhighprob} and dominated convergence:

\begin{lemma}(From $\varepsilon N$ to $ N$ with high probability) \label{Lemma hitting N}\qquad \\
Let $\mathcal{X}$ be a Cannings frequency process with $X_0= k \geq \eps N$ for some $0<\eps<1/2$. Assume  that Conditions \eqref{Dirichlettypeweights}, \eqref{Generating function Y} and \eqref{Condition s_N} are fulfilled. Define $\tau_3: =\inf\{g \geq 0: X_g \in \{0,N \} \}$. Then
	\begin{align}
	\PP_k(X_{\tau_3} =N) = 1-o(1).  
	\end{align}
\end{lemma}

\section{Discussion}
The analysis of fixation probabilities of slightly beneficial mutants is at the heart of population genetics; some seminal and more modern references are given in the Introduction. Our main result concerns Haldane's asymptotics \eqref{Haldaneasymptotic} for the fixation probability in a regime of moderate selection.

Our framework is that of Cannings models with selection (as reviewed in Section \ref{CanningsWithSelection}), where the corresponding neutral genealogies are assumed to belong to the domain of attraction of Kingman's coalescent. 
This class of models is motivated by seasonal reproduction cycles in which within each season a large number of offspring is generated but  only a comparatively small number (concentrated around a carrying capacity $N$) of randomly sampled offspring survive to the next season. In this setting it is reasonable to approximate sampling without replacement by sampling with replacement. Thus, under the assumption of neutrality, the probability that the $j$-th offspring that survives till the next generation is a child of parent $i$ is approximately given by the random weight
\begin{align}
 W_i=\frac{Y_i}{\sum_{\ell=1}^{N} Y_{\ell}},
\end{align}
where $Y_1,\dots,Y_N$ are the sizes of (potential one-generation) offspring of parents $1,\dots,N$. These sizes are assumed to be independent and identically distributed in the present paper, leading to the concept of weights {\em of Dirichlet type}. The subsequent generation then arises by a multinomial sampling with random weights, and to add selection the weights of wildtype parents are decreased by the factor $(1-s_N)$.  For a closely related model with a specific distribution of $Y_i$ (and sampling without replacement) in the context of Lenski’s long-term evolution experiment see \cite{GKWY} and \cite{BGPW}.

We prove Haldane's asymptotics in the case of {\em moderately strong selection}, see Theorem \ref{HaldaneThm},
in which the selection strength $s_N$ obeys
\begin{align}%
	N^{-\frac{1}{2}+\eta} \le s_N \le N^{-\eta}
\end{align}
for some $\eta >0$ and a large population size $N$.
 In the companion paper \cite{BoeGoPoWa1} the range of {\em moderately weak selection} was considered, i.e. in the case
 \begin{align}%
	N^{-1+\eta} \le s_N \le N^{-\frac{1}{2} -\eta}
\end{align}
for some $\eta >0.$ 
Since $s_N \gg N^{-1}$ in the regime of moderate selection, selection acts in this case on a faster timescale than genetic drift.

In \cite{BoeGoPoWa1} an ancestral selection graph for the just described class of Cannings models with selection was defined, and it was shown that the fixation probability $\pi_N$ is equal to the expeced value $\mathbb E \left[\frac{A_{\rm eq}^{(N)}}N\right]$, where $A_{\rm eq}^{(N)}$ is the number of lines of the ancestral selection graph in its equilibrium. While we could analyse directly the asymptotics of that quantity in the regime of moderately weak selection, we were facing too large fluctuations of $A_{\rm eq}^N$ in the regime of moderately strong selection in order to be successful with this approch. Conversely, it turned out that the classical idea of branching process approximation is suitable precisely in that latter regime. 

For highly skewed offspring distributions an asymptotics for the fixation probability arises which is different from \eqref{Haldaneasymptotic}. In cases where the neutral genealogy is attracted by a  Beta$(2-\alpha,\alpha)$-coalescent, \cite{Okada2021} argue that the fixation probability is proportionally to $s_N^{\frac{1}{\alpha-1}}$, if $1\gg s_N \gg N^{-(\alpha-1)}$. Thus  the probability of fixation is substantially smaller than in Haldane's asymptotics, which is reasonable since the offspring variance is diverging as $N\to \infty$. Notably, since the evolutionary timescale of Cannings models in the domain of attraction of Beta-coalescents is of the order $N^{\alpha-1}$, the case $1\gg s_N \gg N^{-(\alpha-1)}$ corresponds to the regime of moderate selection; note also that the case of coalescents being in the domain of attraction of a Kingman coalescent corresponds formally to $\alpha =2$.

\section*{Acknowledgements}
A.G.C. acknowledges CONACyT Ciencia Basica
A1-S-14615 and travel support through the
SPP 1590 and C.P. and A.W. acknowledge partial support through DFG
project WA 967/4-2 in the SPP 1590.
We thank G\"otz Kersting for fruitful discussions, and we are grateful to two anonymous reviewers for their very helpful comments.

 \section*{Conflict of interest}
 The authors declare that they have no conflict of interest.

\bibliography{mybib}
\bibliographystyle{apalike}

\end{document}